\documentclass[12pt,reqno]{amsart}
\usepackage{a4}
\usepackage{amsmath}
\usepackage{amssymb}
\usepackage{amscd}                  
\usepackage{verbatim}

\newcounter{theorem}

\newtheorem{cor}[theorem]{Corollary}
\newtheorem{lem}[theorem]{Lemma}

\newtheorem{pro}[theorem]{Proposition}

\newtheorem{rem}[theorem]{Remark}

\newcounter{intro}

\newtheorem{theo}[intro]{Theorem} 
\newtheorem{prop}[intro]{Proposition}
\newtheorem{coro}[intro]{Corollary}
\newcounter{remintro}

\newtheorem{remi}[remintro]{Remark 1.\!\!}

\numberwithin{equation}{section}

\newcommand{\cref}[1]{Corollary~\ref{#1}}

\newcommand{\pref}[1]{Proposition~\ref{#1}}

\newcommand{\tref}[1]{Theorem~\ref{#1}}

\newcommand{\iprod}[3][{}]{\langle{#2},{#3}\rangle_{#1}}  
\newcommand{\bigiprod}[3][{}]{\bigl\langle{#2},{#3}\bigr\rangle_{#1}}

\newcommand{\tir}{\discretionary{.}{}{---\kern.7em}}

\newcommand\R{\mathbb R}

\newcommand{\Sphere}{\mathbb S} 
\newcommand{\Sn}{{\mathbb S}^n}
  \newcommand\mcC{\mathcal C}

\newcommand\mcD{\mathcal D}

\newcommand{\dom}{\operatorname{dom}}
\newcommand{\Dom}{\operatorname{Dom}}
\newcommand\bargamma{{\bar{\gamma}}}
\newcommand\la{\langle} \newcommand\ra{\rangle}
\newcommand{\eps}{\varepsilon} 
\newcommand{\e}{\mathrm e} 
\renewcommand{\phi}{\varphi}   
\newcommand{\vol}{\text{vol}}

\renewcommand{\Im}{\operatorname{Im}}
\newcommand{\Ker}{\operatorname{Ker}}

\begin{document}
\title[collapsing of connected sums]
{$p$-spectrum and collapsing of connected sums \\
\vspace{0.3cm}
 {\small  ($p$-spectre et effondrement de sommes connexes) }
}

\date{\today}

\author{Colette ANN\'E}
\address{Laboratoire de Math\'eatiques Jean Leray, CNRS-Universit\'e de Nantes,  Facult\'e des Sciences, BP 92208, 44322 Nantes, France}
\email{colette.anne@univ-nantes.fr}
\author{Junya TAKAHASHI}
\address{Division of Mathematics, Graduate School of Information Sciences, 
  T\^ohoku University, Aoba 6-3-09, Sendai, 980-8579, Japan}
\email{junya@math.is.tohocu.ac.jp}
\vspace{3cm}
\begin{abstract}
 The goal of the present paper is to calculate the limit of spectrum of the 
 Hodge-de Rham operator under the perturbation of collapse of one part of 
 a connected sum. It takes place in the general problem of blowing up
conical singularities introduced in \cite{Maz} and \cite{Row}.

Le but de ce travail est de calculer la limite du spectre de l'op\'eateur 
 de Hodge-de Rham dans la perturbation obtenue par effondrement d'une moiti\'e
  d'une somme connexe. 
Ce probl\`eme rentre dans le cadre de l'{\'e}clatement des singularit\'es
 coniques introduites dans \cite{Maz} et \cite{Row}.
\end{abstract}

\footnote[0]{$2000$ {\it Mathematics Subject Classification.} Primary $58J50$; Secondary $35P15,$ $53C23,$ $58J32$. 
 } 


\maketitle

\section{Introduction}

It is a common problem in differential geometry to study the limit of spectrum
of Laplace type operators under singular perturbations of the metrics, 
 especially for the Hodge-de Rham operator acting on differential forms.
 The first reason is the topological meaning of this operator and the fact that by singular perturbations of the metric one can change the topology of the 
manifold. Among a lot of works on this subject we must recall  the study of the
adiabatic limit, started by Mazzeo and Melrose in \cite{MM} and developped by
many authors, we can mention also the collapse of  thin handles started in 
\cite{AC2} and accomplished in \cite{ACP}.

The singular perturbation we study here takes place in the general framework
 of resolution blowups presented in \cite{Maz}, our situation is the collapse 
 of one part in a connected sum, which is explained by the following figure.

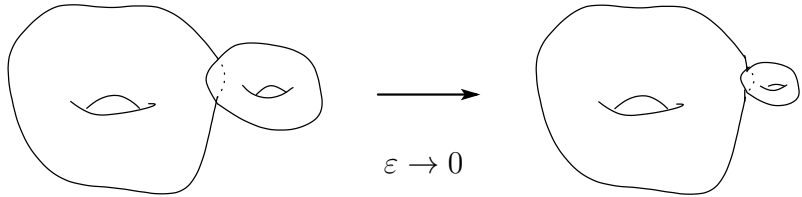
\begin{figure}[h]  \label{fig:conn-sum}

\vspace*{0.5cm}


\unitlength 0.1in
\begin{picture}( 41.4200,  9.9300)( 10.9800,-23.5200)
\special{pn 8}%
\special{pa 2194 1638}%
\special{pa 2174 1612}%
\special{pa 2156 1588}%
\special{pa 2136 1562}%
\special{pa 2114 1538}%
\special{pa 2092 1514}%
\special{pa 2070 1490}%
\special{pa 2046 1468}%
\special{pa 2020 1446}%
\special{pa 1994 1426}%
\special{pa 1966 1408}%
\special{pa 1938 1394}%
\special{pa 1910 1382}%
\special{pa 1880 1372}%
\special{pa 1848 1368}%
\special{pa 1818 1366}%
\special{pa 1786 1366}%
\special{pa 1754 1366}%
\special{pa 1722 1368}%
\special{pa 1690 1372}%
\special{pa 1656 1374}%
\special{pa 1624 1374}%
\special{pa 1590 1372}%
\special{pa 1556 1370}%
\special{pa 1524 1366}%
\special{pa 1490 1362}%
\special{pa 1458 1360}%
\special{pa 1426 1360}%
\special{pa 1394 1362}%
\special{pa 1364 1366}%
\special{pa 1336 1376}%
\special{pa 1308 1390}%
\special{pa 1282 1406}%
\special{pa 1256 1426}%
\special{pa 1230 1448}%
\special{pa 1206 1472}%
\special{pa 1184 1496}%
\special{pa 1162 1520}%
\special{pa 1142 1546}%
\special{pa 1126 1574}%
\special{pa 1114 1602}%
\special{pa 1104 1632}%
\special{pa 1100 1662}%
\special{pa 1098 1694}%
\special{pa 1100 1726}%
\special{pa 1102 1760}%
\special{pa 1106 1792}%
\special{pa 1112 1826}%
\special{pa 1120 1858}%
\special{pa 1126 1890}%
\special{pa 1134 1920}%
\special{pa 1142 1952}%
\special{pa 1152 1982}%
\special{pa 1164 2012}%
\special{pa 1176 2040}%
\special{pa 1190 2070}%
\special{pa 1206 2098}%
\special{pa 1224 2126}%
\special{pa 1244 2152}%
\special{pa 1266 2178}%
\special{pa 1288 2202}%
\special{pa 1312 2224}%
\special{pa 1338 2244}%
\special{pa 1364 2260}%
\special{pa 1392 2274}%
\special{pa 1420 2286}%
\special{pa 1450 2294}%
\special{pa 1480 2300}%
\special{pa 1510 2306}%
\special{pa 1542 2310}%
\special{pa 1574 2312}%
\special{pa 1608 2314}%
\special{pa 1642 2318}%
\special{pa 1676 2322}%
\special{pa 1712 2326}%
\special{pa 1746 2332}%
\special{pa 1782 2338}%
\special{pa 1818 2342}%
\special{pa 1852 2348}%
\special{pa 1886 2350}%
\special{pa 1918 2350}%
\special{pa 1948 2348}%
\special{pa 1978 2344}%
\special{pa 2004 2334}%
\special{pa 2028 2320}%
\special{pa 2048 2302}%
\special{pa 2066 2278}%
\special{pa 2082 2252}%
\special{pa 2096 2224}%
\special{pa 2108 2192}%
\special{pa 2120 2158}%
\special{pa 2130 2124}%
\special{pa 2138 2090}%
\special{pa 2146 2056}%
\special{pa 2154 2024}%
\special{pa 2160 1992}%
\special{pa 2168 1960}%
\special{pa 2174 1928}%
\special{pa 2182 1898}%
\special{pa 2188 1868}%
\special{pa 2194 1842}%
\special{sp}%
\special{pn 8}%
\special{pa 2194 1878}%
\special{pa 2176 1850}%
\special{pa 2160 1822}%
\special{pa 2146 1794}%
\special{pa 2134 1762}%
\special{pa 2136 1732}%
\special{pa 2148 1704}%
\special{pa 2166 1678}%
\special{pa 2192 1652}%
\special{pa 2194 1650}%
\special{sp}%
\special{pn 8}%
\special{pa 2194 1650}%
\special{pa 2216 1626}%
\special{pa 2240 1604}%
\special{pa 2266 1584}%
\special{pa 2292 1570}%
\special{pa 2322 1564}%
\special{pa 2354 1560}%
\special{pa 2386 1560}%
\special{pa 2420 1562}%
\special{pa 2456 1566}%
\special{pa 2490 1572}%
\special{pa 2524 1578}%
\special{pa 2558 1586}%
\special{pa 2590 1596}%
\special{pa 2620 1610}%
\special{pa 2646 1624}%
\special{pa 2670 1644}%
\special{pa 2688 1666}%
\special{pa 2704 1692}%
\special{pa 2714 1722}%
\special{pa 2724 1754}%
\special{pa 2730 1788}%
\special{pa 2736 1822}%
\special{pa 2742 1856}%
\special{pa 2742 1888}%
\special{pa 2736 1918}%
\special{pa 2722 1944}%
\special{pa 2700 1966}%
\special{pa 2670 1984}%
\special{pa 2638 1998}%
\special{pa 2606 2008}%
\special{pa 2574 2014}%
\special{pa 2542 2016}%
\special{pa 2508 2016}%
\special{pa 2476 2014}%
\special{pa 2446 2008}%
\special{pa 2414 2000}%
\special{pa 2384 1988}%
\special{pa 2356 1976}%
\special{pa 2326 1962}%
\special{pa 2298 1946}%
\special{pa 2272 1930}%
\special{pa 2244 1912}%
\special{pa 2218 1894}%
\special{pa 2194 1878}%
\special{sp}%
%
\special{pn 8}%
\special{pa 1426 1866}%
\special{pa 1452 1886}%
\special{pa 1480 1904}%
\special{pa 1506 1918}%
\special{pa 1534 1930}%
\special{pa 1564 1936}%
\special{pa 1598 1938}%
\special{pa 1640 1934}%
\special{pa 1690 1928}%
\special{pa 1740 1920}%
\special{pa 1788 1910}%
\special{pa 1828 1900}%
\special{pa 1858 1890}%
\special{pa 1870 1882}%
\special{pa 1864 1878}%
\special{pa 1834 1878}%
\special{pa 1834 1878}%
\special{sp}%
\special{pn 8}%
\special{pa 1510 1902}%
\special{pa 1534 1882}%
\special{pa 1562 1862}%
\special{pa 1590 1848}%
\special{pa 1620 1838}%
\special{pa 1652 1834}%
\special{pa 1684 1836}%
\special{pa 1714 1848}%
\special{pa 1742 1864}%
\special{pa 1768 1884}%
\special{pa 1786 1902}%
\special{sp}%
\special{pn 8}%
\special{pa 2194 1638}%
\special{pa 2212 1664}%
\special{pa 2226 1692}%
\special{pa 2232 1724}%
\special{pa 2232 1758}%
\special{pa 2224 1788}%
\special{pa 2212 1818}%
\special{pa 2194 1846}%
\special{pa 2182 1866}%
\special{sp -0.045}%
\special{pn 8}%
\special{pa 4958 1640}%
\special{pa 4938 1614}%
\special{pa 4920 1590}%
\special{pa 4900 1564}%
\special{pa 4878 1540}%
\special{pa 4856 1516}%
\special{pa 4834 1492}%
\special{pa 4810 1470}%
\special{pa 4784 1448}%
\special{pa 4758 1428}%
\special{pa 4730 1410}%
\special{pa 4702 1396}%
\special{pa 4674 1384}%
\special{pa 4644 1374}%
\special{pa 4612 1370}%
\special{pa 4582 1368}%
\special{pa 4550 1368}%
\special{pa 4518 1368}%
\special{pa 4486 1370}%
\special{pa 4454 1374}%
\special{pa 4420 1376}%
\special{pa 4388 1376}%
\special{pa 4354 1374}%
\special{pa 4320 1372}%
\special{pa 4288 1368}%
\special{pa 4254 1364}%
\special{pa 4222 1362}%
\special{pa 4190 1362}%
\special{pa 4158 1364}%
\special{pa 4128 1368}%
\special{pa 4100 1378}%
\special{pa 4072 1392}%
\special{pa 4046 1408}%
\special{pa 4020 1428}%
\special{pa 3994 1450}%
\special{pa 3970 1474}%
\special{pa 3948 1498}%
\special{pa 3926 1522}%
\special{pa 3906 1548}%
\special{pa 3890 1576}%
\special{pa 3878 1604}%
\special{pa 3868 1634}%
\special{pa 3864 1664}%
\special{pa 3862 1696}%
\special{pa 3864 1728}%
\special{pa 3866 1762}%
\special{pa 3870 1794}%
\special{pa 3876 1828}%
\special{pa 3884 1860}%
\special{pa 3890 1892}%
\special{pa 3898 1922}%
\special{pa 3906 1954}%
\special{pa 3916 1984}%
\special{pa 3928 2014}%
\special{pa 3940 2042}%
\special{pa 3954 2072}%
\special{pa 3970 2100}%
\special{pa 3988 2128}%
\special{pa 4008 2154}%
\special{pa 4030 2180}%
\special{pa 4052 2204}%
\special{pa 4076 2226}%
\special{pa 4102 2246}%
\special{pa 4128 2262}%
\special{pa 4156 2276}%
\special{pa 4184 2288}%
\special{pa 4214 2296}%
\special{pa 4244 2302}%
\special{pa 4274 2308}%
\special{pa 4306 2312}%
\special{pa 4338 2314}%
\special{pa 4372 2316}%
\special{pa 4406 2320}%
\special{pa 4440 2324}%
\special{pa 4476 2328}%
\special{pa 4510 2334}%
\special{pa 4546 2340}%
\special{pa 4582 2344}%
\special{pa 4616 2350}%
\special{pa 4650 2352}%
\special{pa 4682 2352}%
\special{pa 4712 2350}%
\special{pa 4742 2346}%
\special{pa 4768 2336}%
\special{pa 4792 2322}%
\special{pa 4812 2304}%
\special{pa 4830 2280}%
\special{pa 4846 2254}%
\special{pa 4860 2226}%
\special{pa 4872 2194}%
\special{pa 4884 2160}%
\special{pa 4894 2126}%
\special{pa 4902 2092}%
\special{pa 4910 2058}%
\special{pa 4918 2026}%
\special{pa 4924 1994}%
\special{pa 4932 1962}%
\special{pa 4938 1930}%
\special{pa 4946 1900}%
\special{pa 4952 1870}%
\special{pa 4958 1844}%
\special{sp}%
\special{pn 8}%
\special{pa 4190 1868}%
\special{pa 4216 1888}%
\special{pa 4244 1906}%
\special{pa 4270 1920}%
\special{pa 4298 1932}%
\special{pa 4328 1938}%
\special{pa 4362 1940}%
\special{pa 4404 1936}%
\special{pa 4454 1930}%
\special{pa 4504 1922}%
\special{pa 4552 1912}%
\special{pa 4592 1902}%
\special{pa 4622 1892}%
\special{pa 4634 1884}%
\special{pa 4628 1880}%
\special{pa 4598 1880}%
\special{pa 4598 1880}%
\special{sp}%
\special{pn 8}%
\special{pa 4274 1904}%
\special{pa 4298 1884}%
\special{pa 4326 1864}%
\special{pa 4354 1850}%
\special{pa 4384 1840}%
\special{pa 4416 1836}%
\special{pa 4448 1838}%
\special{pa 4478 1850}%
\special{pa 4506 1866}%
\special{pa 4532 1886}%
\special{pa 4550 1904}%
\special{sp}%
%
\special{pn 13}%
\special{pa 3034 1830}%
\special{pa 3550 1830}%
\special{fp}%
\special{sh 1}%
\special{pa 3550 1830}%
\special{pa 3482 1810}%
\special{pa 3496 1830}%
\special{pa 3482 1850}%
\special{pa 3550 1830}%
\special{fp}%
\special{pn 8}%
\special{pa 4966 1706}%
\special{pa 4988 1684}%
\special{pa 5016 1668}%
\special{pa 5046 1662}%
\special{pa 5078 1662}%
\special{pa 5114 1666}%
\special{pa 5148 1674}%
\special{pa 5178 1686}%
\special{pa 5204 1702}%
\special{pa 5220 1728}%
\special{pa 5230 1758}%
\special{pa 5236 1792}%
\special{pa 5240 1824}%
\special{pa 5230 1852}%
\special{pa 5204 1872}%
\special{pa 5172 1884}%
\special{pa 5140 1888}%
\special{pa 5108 1888}%
\special{pa 5076 1880}%
\special{pa 5046 1868}%
\special{pa 5018 1854}%
\special{pa 4990 1836}%
\special{pa 4966 1820}%
\special{sp}%
%
\special{pn 8}%
\special{pa 4966 1820}%
\special{pa 4966 1820}%
\special{pa 4966 1820}%
\special{pa 4966 1820}%
\special{sp}%
%
\special{pn 8}%
\special{pa 4966 1710}%
\special{pa 4964 1678}%
\special{pa 4958 1646}%
\special{pa 4954 1626}%
\special{sp}%
%
\special{pn 8}%
\special{pa 4954 1866}%
\special{pa 4954 1834}%
\special{pa 4954 1806}%
\special{sp}%
\special{pn 8}%
\special{pa 4966 1818}%
\special{pa 4950 1788}%
\special{pa 4934 1760}%
\special{pa 4932 1732}%
\special{pa 4950 1708}%
\special{pa 4978 1686}%
\special{sp}%
\special{pn 8}%
\special{pa 4966 1686}%
\special{pa 4988 1708}%
\special{pa 5002 1738}%
\special{pa 5002 1770}%
\special{pa 4986 1798}%
\special{pa 4966 1818}%
\special{sp -0.045}%
\put(32.7300,-21.8900){\makebox(0,0){$\varepsilon \to 0$}}%
\special{pn 8}%
\special{pa 2326 1782}%
\special{pa 2348 1806}%
\special{pa 2372 1824}%
\special{pa 2402 1838}%
\special{pa 2434 1846}%
\special{pa 2468 1848}%
\special{pa 2500 1842}%
\special{pa 2530 1832}%
\special{pa 2558 1816}%
\special{pa 2584 1796}%
\special{pa 2590 1794}%
\special{sp}%
\special{pn 8}%
\special{pa 2398 1830}%
\special{pa 2424 1806}%
\special{pa 2450 1788}%
\special{pa 2478 1782}%
\special{pa 2506 1790}%
\special{pa 2536 1806}%
\special{pa 2554 1818}%
\special{sp}%
\special{pn 8}%
\special{pa 5036 1786}%
\special{pa 5066 1802}%
\special{pa 5096 1810}%
\special{pa 5126 1806}%
\special{pa 5154 1792}%
\special{pa 5174 1780}%
\special{sp}%
\special{pn 8}%
\special{pa 5078 1786}%
\special{pa 5108 1776}%
\special{pa 5138 1778}%
\special{pa 5162 1786}%
\special{sp}%
\end{picture}%

 \begin{center}
  \caption[collapsing of $M_{\varepsilon}$]{collapsing of $M_{\varepsilon}$}
 \end{center}
\end{figure}

\vspace{- 0.8cm}

More precisely, if the manifold $M$, of dimension $m$, is the connected sum 
of $M_1$ and $M_2$ around the common point $p_0$, endowed with Riemannian 
metrics $g_1, g_2,$ then, for the collapse of one part of the connected 
 sum we study the dependence on $\eps\to 0$ for the manifold 
 $$ M_{\eps} := (M_1-B(p_0,\eps)) \cup \eps.(M_2-B(p_0,1)), $$ 
 where $\eps.(M_2 - B(p_0,1) )$ means $( M_2 - B(p_0,1), \eps^2 g_2)$.

To make this construction clear, we can suppose that the two metrics are flat 
around  the point $p_0,$ then the boundaries of $(M_1-B(p_0,\eps), g_1)$ and 
 $(M_2 - B(p_0,1), \eps^2 g_2)$ are isometric, and can be identified. One can 
then define geometrically  $M_{\eps}$ as a Riemannian manifold $C^\infty$ by 
part.

In the terminology of \cite{Maz}, $M_{\eps}$ is the resolution blowup of
 the (singular) space $M_1$ with $\Sn$ as {\em link} and $\widetilde{M}_2$ 
as {\em asymptotically conical manifold}, if $\widetilde{M}_2$ 
is the complete manifold obtained by gluing of the exterior of a ball in the 
 Euclidean space  to the boundary of $M_2 - B(p_0,1).$
In fact, Rowlett studies, in \cite{Row}, the convergence of the spectrum 
of generalized Laplacian on such a situation of blowup of one isolated 
conical singularity (Mazzeo presents more genral sinbularities in \cite{Maz}). 
Her result gives the convergence of the spectrum to
the spectrum of an operator on $M_1$, it requires an hypothesis on 
 $\widetilde{M}_2$. Our result is less general, applied only to the case of 
 the Hodge-de Rham operator, but it does not require this hypothesis and the 
limit spectrum takes care of $\widetilde{M}_2$, see Theorem \ref{C} below.

 Maybe, the more important interest of our study is that we introduce new 
 techniques: to solve this kind of problems, we have to identify a good 
 elliptic limit problem, this means for the $M_2$ part a good boundary 
problem on $M_2 - B(p_0,1)$ at the limit. It appears that, on difference 
 with the problem of thin handles in \cite{AC2} or the connected sum problem 
studied in \cite{T} for functions, this boundary problem is not a kind of local  but `global': we have to introduce a condition of the Atiyah-Patodi-Singer 
 (APS) type, as defined in \cite{APS}.

 Indeed, these APS boundary conditions are related to the Fredholm theory
 on the link $\widetilde{M}_2$, as explained by Carron in \cite{C}, details 
 are given below.

 We shall study the more general blowup of conical spaces in the future.
\subsection{The results} 
 As mentioned above, the manifold $M$, of dimension $m \geq 3$ (there is no 
 problem in dimension 2), is the connected sum of two Riemannian manifolds 
 $(M_1,g_1)$ and $(M_2,g_2)$ around the common point $p_0,$  and we suppose 
 that the metrics are such that the boundaries of $(M_1 - B(p_0,\eps), g_1)$ 
 and $(M_2 - B(p_0,1), \eps^2 g_2)$ are isometric for all $\eps$ small enough. 
 As a consequence, $(M_1,g_1)$ is flat in a neighborhood of $p_0$ and 
 $\partial(M_2-B(p_0,1))$ is the standard sphere.
 Indeed one can write $g_1 = dr^2 + r^2 h(r)$ in the polar coordinates around 
$p_0\in M_1$ and the metric $h(r)$ on the sphere converges, as $r \to 0,$ to 
 the standard metric. 
 But if the boundaries of $(M_1 - B(p_0, \eps), g_1)$ and 
 $(M_2-B(p_0,1), \eps^2 g_2)$ are isometric for all $\eps$ small enough, then 
 $h(r)$ is constant for $r$ small enough, the conclusion follows.

 One can then define geometrically 
  $M_{\eps} := (M_1 - B(p_0,\eps)) \cup \eps. (M_2 - B(p_0,1))$
  as the connected sum obtained by the collapse of $M_2$ (the question of the 
metric on $M_{\eps}$ is discussed below).
 On such a manifold, the Gau{\ss}-Bonnet operator $D_\eps,$  Sobolev spaces  
and also the Hodge-de Rham operator $\Delta_\eps$ can be defined as follows 
 (the details are given in \cite{AC2}): on a manifold $X=X_1\cup X_2,$ 
 which is the union of two Riemannian manifolds with isometric boundaries, 
 if $D_1$ and $D_2$ are the Gau{\ss}-Bonnet ``$d+d^\ast$'' operators acting on
  the differential forms of each part, the quadratic form 
 $$ q(\phi) = \displaystyle \int_{X_1} | D_1(\phi \upharpoonright_{X_1})|^2
       \, d \mu_{X_1} + 
     \int_{X_2}  | D_2(\phi \upharpoonright_{X_2})|^2 \, d \mu_{X_2} $$
 is well-defined and closed on the domain
\begin{equation*}
 \Dom (q) := \{ \phi = (\phi_1,\phi_2) \in H^1(\Lambda T^\ast X_{1}) \times
   H^1(\Lambda T^\ast X_{2}) \, | \,
  \phi_1 \upharpoonright_{\partial X_1} \stackrel{L_2}{=}
  \phi_2 \upharpoonright_{\partial X_2}  \},
\end{equation*}
 where the boundary values $ \phi_i \upharpoonright_{\partial X_i}$ are 
considered in the sense of the trace operator, and on this space the total 
Gau\ss-Bonnet operator $D(\phi)=(D_1(\phi_1),D_2(\phi_2)) $ is defined and 
selfadjoint.
  For this definition, we have, in particular, to identify 
 $(\Lambda T^\ast X_{1}) \upharpoonright_{\partial X_1}$ and
 $(\Lambda T^\ast X_{2}) \upharpoonright_{\partial X_2}.$
  This can be done by decomposing the forms in tangential and normal parts 
(with inner normal), the equality above means then that the tangential parts
 are equal and the normal parts opposite. This definition generalizes the 
definition in the smooth case.

The Hodge-de Rham operator $(d+d^\ast)^2$ of $X$ is then defined as the 
operator obtained by the polarization of the quadratic form $q$. This gives 
compatibility conditions between $\phi_1$ and $\phi_2$ on the common boundary.
  We do not give details on these facts because, as remarked in the next 
section, it is sufficient to work with smooth metrics on $M.$

 The multiplicity of $0$ in the spectrum of $\Delta_\eps$ is given by the 
cohomology, it is  then independent of $\eps$ and can be related to the 
cohomology of each part by the Mayer-Vietoris argument. The point is to study 
the convergence of the other eigenvalues, the so-called 
{\em positive spectrum}, as $\eps \to 0.$
 The second author has shown in \cite{T2}, Theorem $4.4$, p.$21$, a result of boundedness
\begin{prop}[Takahashi]\label{Taka}
 The superior limit of the $k$-th {\em positive } eigenvalue on $p$-forms of 
 $M_{\eps}$ is bounded, as $\eps \to 0$, by the $k$-th {\em positive } 
 eigenvalue on $p$-forms of $M_1.$ 
\end{prop}

We show here that it is also true for the lower bound. 
Let $\phi_\eps $ be a family of eigenforms on $M_{\eps}$ of degree $p$ for 
the Hodge-de Rham operator:
$$
 \Delta_\eps \phi_\eps = \lambda^p (M_{\eps}) \phi_\eps \; ;\;
  \lim_{\eps \to 0} \lambda^p (M_{\eps}) = \lambda^p < + \infty.
$$
\begin{prop}\label{B}
 If $\lambda^p (M_{\eps}) \neq 0,$ then $\lambda^p \neq 0$ and $\lambda^p$
  belongs to the spectrum of the Hodge-de Rham operator on $(M_1, g_1).$
\end{prop}

 The first point is a consequence of the application of the so-called 
 McGowan's  lemma; indeed $M_{\eps}$ has no small eigenvalues as is shown in 
Proposition \ref{McG} below.
 To prove the convergent part of the proposition, we shall decompose the 
eigenforms using the good control of the APS-boundary term.
 More precisely, there exists an elliptic extension $\mcD_2$ of the 
Gau\ss-Bonnet operator $D_2$ on $M_2(1) = M_2 - B(p_0,1)$ and a family 
$\psi_\eps$ bounded in  $H^1(M_1) \times \hbox{Dom} (\mcD_2)$ such that 
  $ \| \phi_\eps - \psi_\eps \| \to 0$ with $\eps$.

This extension is defined by {\em global} boundary conditions, the conditions 
of APS type, in relation with the works of Carron about operators non parabolic 
at infinity developped in \cite{C}, see proposition \ref{prop:D_2-elliptic}.

 If we make this construction for an orthonormal family of the $k$ first 
 eigenforms, we obtain, with the help of Proposition \ref{Taka}, 
 our main theorem
\begin{theo}\label{C}
  Let $M_{\eps} =(M_1 - B(p_0,\eps)) \cup \eps.(M_2 - B(p_0,1))$ be the 
 connected sum of the two Riemannian manifolds $M_1$ and $\eps. M_2$ of 
 dimension $m = n+1.$
 For $p \in \{ 1, \dots, n \},$
  let $ 0 < \lambda^p_1(M_1) \leq \lambda^p_2(M_1), \dots $ be the {\em positive} spectrum of the Hodge-de Rham operator on the $p$-forms of $M_1$ and 
  $0 < \lambda^p_1(M_{\eps}) \leq \lambda^p_2(M_{\eps}), \dots $ 
 the {\em positive} spectrum of the Hodge-de Rham operator on the $p$-forms of
  $M_{\eps}.$ Then, for all $k \ge 1$ we obtain 
 $$
     \lim_{\eps \to 0} \lambda^p_k (M_{\eps}) = \lambda^p_k (M_1).
 $$
 Moreover, the multiplicity of $0$ is given by the cohomology and
 $$ H^p( M_{\eps}; {\mathbb R} ) \cong
     H^p(M_1; {\mathbb R}) \oplus H^p(M_2; {\mathbb R}).
 $$
\end{theo}
\begin{remi}
 The result of convergence of the positive spectrum is also true for $p=0$ and 
 has been shown in \cite{T}. 
 Naturally $H^0 (M_{\eps}; \R) \cong H^0(M_1; \R)= \R.$ 
 By the Hodge duality this solves also the case $p=m.$
\end{remi}
\subsection{Applications}
 Results on spectral convergence in singular situations can be used
to give examples or counter examples, concerning possible links
 between spectral and geometric properties. 
 For instance, Colbois and El Soufi have introduced in \cite{CE2} 
 the notion of {\em conformal spectrum} as the supremum, for each integer $k$,
 of the value of the  $k$-th eigenvalue on a conformal class of metrics with
  fixed volume. Using the result of \cite{T}, they could show that 
 the conformal spectrum of a compact manifold is always bounded from below by 
 the conformal spectrum of the standard sphere of the same dimension.

 In the same way, applying the \tref{C} to the case $M_1 = \Sphere^m$ and 
$M_2=M,$ we obtain
\begin{coro}
 Let $(M,g)$ be a compact Riemannian manifold of dimension $m$, for any degree 
 $p$, any integer $N \geq 1$ and any $\eps > 0,$ there exists on $M$ a metric 
 $\overline g$ conformal to $g$ such that the $N$ first positive eigenvalues 
 on the $p$-forms are $\eps$-close to the $N$ first positive eigenvalues 
 on the $p$-forms of the standard sphere with the same dimension and the same 
volume as  $(M,g)$.
\end{coro}
\begin{remi}
 For the completion of the panorama on this subject, let us recall that 
Jammes has shown, in \cite{J}, that in dimension $m \geq 4$ the 
infimum of the $p$-spectrum in a conformal class, with fixed volume, is $0$ 
for $2 \leq p \leq m-2$ and $p \neq \frac{m}{2}$ but has a positive lower 
bound for $p=\frac{m}{2}$.
\end{remi}

 Another example is the {\em Prescription of the spectrum}.
  This question was introduced by Colin de Verdi\`ere in \cite{CdV,CdV1} 
  where he shows that he can impose any finite part of the spectrum of the 
  Laplace-Beltrami operator on certain manifolds. 
  To this goal, he introduced a very powerful technique of transversality,
and shows that this hypothesis is satisfied on certain graphs and on certain 
manifolds \cite{CdV2}. 
  The other necessary argument is a theorem of convergence.
 The solution of the problem of prescription, with limitation concerning 
 multiplicity, has been given by Guerini in \cite{Gu} for the Hodge-de Rham 
 operator, and Jammes has proved a result of prescription, without 
 multiplicity, in a conformal class of the metric in \cite{J2}, 
 for certain degrees of the differential forms, what is compatible with 
 the retricted result mentioned above.  
  In this context, our result gives, for example, 
\begin{coro}
 Let $g_0$ be a metric on the sphere of dimension $m$. If $g_0$  satisfies 
the Strong Arnol'd
  Hypothesis, following the terminology of \cite{CdV}, for the eigenvalue 
  $\lambda \neq 0$ on differential forms of degree $p$ on the sphere, 
  then for any closed manifold $M,$ there exists a metric such that
   $\lambda$ belongs to the spectrum of the Hodge-de Rham operator on 
  $p$-forms with the same multiplicity.
\end{coro}
 Indeed, we take a metric $g_2$ on $M$, and for any metric $g_1$ close to 
 $g_0$,  the positive spectrum of $M_\eps = \Sphere^m \# \eps .M$ converges,
  as $\eps \to 0$ to the spectrum of $\Sphere^m.$
  Then, the Strong Arnol'd Hypothesis assures that the map which associates to
 $g_1$ the {\em spectral quadratic form}  relative to a small interval $I$ 
 around  $\lambda$ has also, for $\eps$ small enough, the matrix 
  $\lambda \cdot \text{Id}$ in its image.

 Here, by spectral quadratic form, we mean the quadratic form defined by
the Hodge-de Rham operator, restricted to the eigenspace of eigenforms 
with eigenvalues in $I$. To consider this space as a space of matrix, we have 
to construct small isometries between the different eigenspaces, the details 
are in \cite{CdV2}.

This result could be used to prescribe high multiplicity for the spectrum of
the Hodge-de Rham operator. Recall that Jammes had obtained partial results
on this subject in \cite{J3}, his work is based on a convergence theorem 
(theorem 2.8) where the limit is the Hodge-de Rham operator with absolute 
boundary condition on a domain, he uses also the fact that the Strong Arnol'd 
Hypothesis is satisfied on spheres of dimension 2, as proved in \cite{CdV2}.
It would be interesting to obtain such a result on spheres of bigger 
dimension, the result of \cite{CdV2} uses the conformal invariance specific
to this dimension.

\subsubsection*{Acknoledgement}
 This work started with a visit of the second author at the Laboratory Jean 
Leray in Nantes. He is grateful to the Laboratory and the University of Nantes 
for hospitality. 
 
 \vspace{0.5cm}
 We now proceed to prove the theorems. Let us first describe the metrics 
precisely. 

\section{Choice of the metric}

 From now on, we denote
  $$ M_2(1) := M_2 - B(p_0,1). $$
 It is supposed here that the ball $B(p_0, 1)$ can really be embedded in the 
 manifold $M_2$, this can always be satisfied by a scaling of the metric $g_2$  on $M_2$.

 Recall that Dodziuk has proved in ~\cite[Prop.~3.3]{D} that 
 if two metrics $g, \; \overline{g}$ on the same compact manifold satisfy
\begin{equation}\label{eq:dod}
  \e^{-\eta} g \le \overline{g} \le \e^\eta g.
\end{equation}
 Then, the corresponding eigenvalues of the Hodge-de Rham operator acting on 
 $p$-forms satisfy
\begin{equation*}
  \e^{ -(n+2p) \eta} \lambda_k^p (g) \le  \lambda_k^p (\overline{g}) \le 
  \e^{ (n+2p) \eta}  \lambda_k^p (g).
\end{equation*}
 This result is based on the fact that the multiplicity of $0$ is given by the 
 cohomology and the positive spectrum by exact forms, hence the min-max formula  does not involve derivatives of the metric; it stays valid if one of the two 
 metrics is only smooth by part, because in the last case the Hodge decomposition still holds true.

 Then, for a metric  $g_1$ on $M_1$ there exists, for each $\eta > 0$ a metric 
 $\overline{g}_1$ on $M_1$ which is flat on a ball $B_\eta$ centered at $p_0$ and such that
$$
  \e^{-\eta} g_1\le \overline{g}_1 \le \e^\eta g_1.
$$
 Then our result can be extended to any other construction which does not suppose that the metric $g_1$ is flat in a neighborhood of $p_0$. 

 Now, we regard $M_{\eps}$ as the union of $M_1 - B(p_0, 3 \eps)$ and 
  $\eps. \overline{M}_2(1)$, where
 $\overline{M}_2(1) = \big( B_{\R^m}( 0,3) - B_{\R^m}(0, 1) \big) \cup M_2(1)$
 is endowed with a metric only smooth by part: the Euclidean metric on the
 first part and the restriction of $g_2$ on the second part.
 But this metric can be approached, as close as we want, by a smooth metric which is still flat  on $B_{\R^m}( 0, 3) - B_{\R^m} (0, \frac{3}{2})$ and these two metrics will satisfy the estimate \eqref{eq:dod}. 
 Thus, replacing $3 \eps$ by $\eps$ for simplicity, we can suppose, without loss of generality, that we are in the following situation:
\vskip5pt
{\em 
 The manifold $M_2(1)$ is endowed with a metric which is conical $($flat$)$ near the boundary,  namely $g_2 = ds^2 + (1-s)^2 h$, $h$ being the canonical metric of   the sphere $\Sn = \partial(M_2(1))$, and $s \in [0, \frac{1}{2} )$ being the distance from the boundary $( M_2(1)$ looks like a trumpet $)$ and 
 $M_1(\eps) =M_1 - B(p_0, \eps)$ with a conical metric $g_1 = dr^2 + r^2 h$ 
 around the point $p_0.$
  Thus, $M_{\eps} = M_1(\eps) \cup \eps.M_2(1)$ is a smooth Riemannian manifold.  }

\begin{figure}[h]  \label{fig:smoothing}


\unitlength 0.1in
\begin{picture}( 19.0100,  9.5700)( 16.6000,-20.8700)
\special{pn 8}%
\special{ar 2530 1580 50 274  4.7711448 6.2831853}%
\special{ar 2530 1580 50 274  0.0000000 4.7123890}%
\special{pn 8}%
\special{ar 2530 1580 50 274  4.7711448 6.2831853}%
\special{ar 2530 1580 50 274  0.0000000 4.7123890}%
\special{pn 8}%
\special{ar 3084 1584 24 202  1.3909428 4.8845798}%
\special{pn 8}%
\special{ar 2530 1580 50 274  4.7711448 6.2831853}%
\special{ar 2530 1580 50 274  0.0000000 4.7123890}%
\special{pn 8}%
\special{ar 2954 1588 38 212  4.7804654 6.2831853}%
\special{ar 2954 1588 38 212  0.0000000 4.7123890}%
\special{pn 8}%
\special{pa 2682 1320}%
\special{pa 2710 1336}%
\special{pa 2738 1350}%
\special{pa 2766 1364}%
\special{pa 2796 1376}%
\special{pa 2826 1384}%
\special{pa 2860 1392}%
\special{pa 2894 1398}%
\special{pa 2926 1396}%
\special{pa 2954 1386}%
\special{pa 2964 1380}%
\special{sp}%
\special{pn 8}%
\special{pa 2534 1306}%
\special{pa 2514 1282}%
\special{pa 2494 1256}%
\special{pa 2472 1232}%
\special{pa 2448 1210}%
\special{pa 2424 1192}%
\special{pa 2396 1174}%
\special{pa 2368 1160}%
\special{pa 2338 1146}%
\special{pa 2308 1134}%
\special{pa 2278 1130}%
\special{pa 2246 1132}%
\special{pa 2214 1134}%
\special{pa 2182 1132}%
\special{pa 2150 1132}%
\special{pa 2118 1134}%
\special{pa 2086 1136}%
\special{pa 2054 1140}%
\special{pa 2022 1146}%
\special{pa 1990 1152}%
\special{pa 1960 1160}%
\special{pa 1928 1170}%
\special{pa 1898 1180}%
\special{pa 1870 1194}%
\special{pa 1842 1208}%
\special{pa 1814 1226}%
\special{pa 1788 1244}%
\special{pa 1762 1262}%
\special{pa 1738 1284}%
\special{pa 1716 1306}%
\special{pa 1696 1334}%
\special{pa 1682 1362}%
\special{pa 1674 1392}%
\special{pa 1668 1424}%
\special{pa 1666 1456}%
\special{pa 1664 1488}%
\special{pa 1664 1520}%
\special{pa 1664 1552}%
\special{pa 1664 1584}%
\special{pa 1664 1616}%
\special{pa 1664 1648}%
\special{pa 1662 1682}%
\special{pa 1662 1714}%
\special{pa 1666 1746}%
\special{pa 1676 1774}%
\special{pa 1692 1800}%
\special{pa 1714 1826}%
\special{pa 1736 1850}%
\special{pa 1758 1878}%
\special{pa 1778 1902}%
\special{pa 1800 1926}%
\special{pa 1824 1944}%
\special{pa 1852 1958}%
\special{pa 1882 1964}%
\special{pa 1914 1968}%
\special{pa 1948 1972}%
\special{pa 1980 1976}%
\special{pa 2012 1982}%
\special{pa 2044 1990}%
\special{pa 2074 1996}%
\special{pa 2106 2002}%
\special{pa 2136 2004}%
\special{pa 2168 2004}%
\special{pa 2200 2004}%
\special{pa 2232 2002}%
\special{pa 2268 2002}%
\special{pa 2302 2002}%
\special{pa 2334 2000}%
\special{pa 2364 1994}%
\special{pa 2386 1982}%
\special{pa 2408 1964}%
\special{pa 2432 1938}%
\special{pa 2466 1902}%
\special{pa 2504 1866}%
\special{pa 2530 1842}%
\special{pa 2530 1848}%
\special{pa 2528 1852}%
\special{sp}%
\special{pn 8}%
\special{ar 2530 1580 50 274  4.7711448 6.2831853}%
\special{ar 2530 1580 50 274  0.0000000 4.7123890}%
\special{pn 8}%
\special{pa 1894 1552}%
\special{pa 1920 1572}%
\special{pa 1946 1588}%
\special{pa 1976 1602}%
\special{pa 2006 1614}%
\special{pa 2036 1622}%
\special{pa 2068 1626}%
\special{pa 2100 1626}%
\special{pa 2132 1628}%
\special{pa 2164 1628}%
\special{pa 2196 1624}%
\special{pa 2228 1616}%
\special{pa 2256 1602}%
\special{pa 2282 1584}%
\special{pa 2306 1556}%
\special{pa 2308 1526}%
\special{pa 2302 1522}%
\special{sp}%
\special{pn 8}%
\special{pa 1966 1596}%
\special{pa 1992 1578}%
\special{pa 2022 1564}%
\special{pa 2052 1556}%
\special{pa 2084 1556}%
\special{pa 2116 1554}%
\special{pa 2148 1552}%
\special{pa 2180 1554}%
\special{pa 2212 1562}%
\special{pa 2240 1564}%
\special{pa 2276 1582}%
\special{pa 2272 1582}%
\special{sp}%
\special{pn 8}%
\special{ar 2530 1580 50 274  4.7711448 6.2831853}%
\special{ar 2530 1580 50 274  0.0000000 4.7123890}%
\special{pn 8}%
\special{ar 2530 1580 50 274  4.7711448 6.2831853}%
\special{ar 2530 1580 50 274  0.0000000 4.7123890}%
\special{pn 8}%
\special{ar 2684 1588 44 276  1.5707963 4.7123890}%
\special{pn 8}%
\special{ar 2530 1580 50 274  4.7711448 6.2831853}%
\special{ar 2530 1580 50 274  0.0000000 4.7123890}%
\special{pn 8}%
\special{ar 2530 1580 50 274  4.7711448 6.2831853}%
\special{ar 2530 1580 50 274  0.0000000 4.7123890}%
\special{pn 8}%
\special{ar 3084 1584 24 202  1.3909428 4.8845798}%
\special{pn 8}%
\special{ar 2530 1580 50 274  4.7711448 6.2831853}%
\special{ar 2530 1580 50 274  0.0000000 4.7123890}%
\special{pn 8}%
\special{ar 2530 1580 50 274  4.7711448 6.2831853}%
\special{ar 2530 1580 50 274  0.0000000 4.7123890}%
%
\special{pn 8}%
\special{ar 2954 1588 38 212  4.7804654 6.2831853}%
\special{ar 2954 1588 38 212  0.0000000 4.7123890}%
\special{pn 8}%
\special{pa 3084 1384}%
\special{pa 3114 1390}%
\special{pa 3146 1394}%
\special{pa 3178 1392}%
\special{pa 3210 1386}%
\special{pa 3240 1376}%
\special{pa 3270 1360}%
\special{pa 3300 1354}%
\special{pa 3332 1358}%
\special{pa 3364 1364}%
\special{pa 3396 1366}%
\special{pa 3426 1378}%
\special{pa 3450 1398}%
\special{pa 3472 1422}%
\special{pa 3492 1448}%
\special{pa 3510 1474}%
\special{pa 3524 1502}%
\special{pa 3536 1532}%
\special{pa 3546 1564}%
\special{pa 3556 1594}%
\special{pa 3560 1626}%
\special{pa 3560 1658}%
\special{pa 3554 1688}%
\special{pa 3542 1718}%
\special{pa 3530 1750}%
\special{pa 3516 1780}%
\special{pa 3498 1806}%
\special{pa 3472 1822}%
\special{pa 3442 1832}%
\special{pa 3408 1836}%
\special{pa 3376 1840}%
\special{pa 3346 1844}%
\special{pa 3314 1848}%
\special{pa 3280 1852}%
\special{pa 3250 1850}%
\special{pa 3222 1836}%
\special{pa 3196 1814}%
\special{pa 3174 1788}%
\special{pa 3150 1766}%
\special{pa 3122 1758}%
\special{pa 3090 1768}%
\special{pa 3078 1774}%
\special{sp}%
%
\special{pn 8}%
\special{ar 2530 1580 50 274  4.7711448 6.2831853}%
\special{ar 2530 1580 50 274  0.0000000 4.7123890}%
\special{pn 8}%
\special{pa 2534 1306}%
\special{pa 2514 1282}%
\special{pa 2494 1256}%
\special{pa 2472 1232}%
\special{pa 2448 1210}%
\special{pa 2424 1192}%
\special{pa 2396 1174}%
\special{pa 2368 1160}%
\special{pa 2338 1146}%
\special{pa 2308 1134}%
\special{pa 2278 1130}%
\special{pa 2246 1132}%
\special{pa 2214 1134}%
\special{pa 2182 1132}%
\special{pa 2150 1132}%
\special{pa 2118 1134}%
\special{pa 2086 1136}%
\special{pa 2054 1140}%
\special{pa 2022 1146}%
\special{pa 1990 1152}%
\special{pa 1960 1160}%
\special{pa 1928 1170}%
\special{pa 1898 1180}%
\special{pa 1870 1194}%
\special{pa 1842 1208}%
\special{pa 1814 1226}%
\special{pa 1788 1244}%
\special{pa 1762 1262}%
\special{pa 1738 1284}%
\special{pa 1716 1306}%
\special{pa 1696 1334}%
\special{pa 1682 1362}%
\special{pa 1674 1392}%
\special{pa 1668 1424}%
\special{pa 1666 1456}%
\special{pa 1664 1488}%
\special{pa 1664 1520}%
\special{pa 1664 1552}%
\special{pa 1664 1584}%
\special{pa 1664 1616}%
\special{pa 1664 1648}%
\special{pa 1662 1682}%
\special{pa 1662 1714}%
\special{pa 1666 1746}%
\special{pa 1676 1774}%
\special{pa 1692 1800}%
\special{pa 1714 1826}%
\special{pa 1736 1850}%
\special{pa 1758 1878}%
\special{pa 1778 1902}%
\special{pa 1800 1926}%
\special{pa 1824 1944}%
\special{pa 1852 1958}%
\special{pa 1882 1964}%
\special{pa 1914 1968}%
\special{pa 1948 1972}%
\special{pa 1980 1976}%
\special{pa 2012 1982}%
\special{pa 2044 1990}%
\special{pa 2074 1996}%
\special{pa 2106 2002}%
\special{pa 2136 2004}%
\special{pa 2168 2004}%
\special{pa 2200 2004}%
\special{pa 2232 2002}%
\special{pa 2268 2002}%
\special{pa 2302 2002}%
\special{pa 2334 2000}%
\special{pa 2364 1994}%
\special{pa 2386 1982}%
\special{pa 2408 1964}%
\special{pa 2432 1938}%
\special{pa 2466 1902}%
\special{pa 2504 1866}%
\special{pa 2530 1842}%
\special{pa 2530 1848}%
\special{pa 2528 1852}%
\special{sp}%
\special{pn 8}%
\special{ar 2530 1580 50 274  4.7711448 6.2831853}%
\special{ar 2530 1580 50 274  0.0000000 4.7123890}%
\special{pn 8}%
\special{pa 1894 1552}%
\special{pa 1920 1572}%
\special{pa 1946 1588}%
\special{pa 1976 1602}%
\special{pa 2006 1614}%
\special{pa 2036 1622}%
\special{pa 2068 1626}%
\special{pa 2100 1626}%
\special{pa 2132 1628}%
\special{pa 2164 1628}%
\special{pa 2196 1624}%
\special{pa 2228 1616}%
\special{pa 2256 1602}%
\special{pa 2282 1584}%
\special{pa 2306 1556}%
\special{pa 2308 1526}%
\special{pa 2302 1522}%
\special{sp}%
\special{pn 8}%
\special{pa 1966 1596}%
\special{pa 1992 1578}%
\special{pa 2022 1564}%
\special{pa 2052 1556}%
\special{pa 2084 1556}%
\special{pa 2116 1554}%
\special{pa 2148 1552}%
\special{pa 2180 1554}%
\special{pa 2212 1562}%
\special{pa 2240 1564}%
\special{pa 2276 1582}%
\special{pa 2272 1582}%
\special{sp}%
\special{pn 8}%
\special{ar 2530 1580 50 274  4.7711448 6.2831853}%
\special{ar 2530 1580 50 274  0.0000000 4.7123890}%
\special{pn 8}%
\special{ar 2530 1580 50 274  4.7711448 6.2831853}%
\special{ar 2530 1580 50 274  0.0000000 4.7123890}%
\special{pn 8}%
\special{ar 2684 1588 44 276  1.5707963 4.7123890}%
\special{pn 8}%
\special{ar 2530 1580 50 274  4.7711448 6.2831853}%
\special{ar 2530 1580 50 274  0.0000000 4.7123890}%
\special{pn 8}%
\special{ar 2530 1580 50 274  4.7711448 6.2831853}%
\special{ar 2530 1580 50 274  0.0000000 4.7123890}%
\special{pn 8}%
\special{ar 3084 1584 24 202  1.3909428 4.8845798}%
\special{pn 8}%
\special{ar 2530 1580 50 274  4.7711448 6.2831853}%
\special{ar 2530 1580 50 274  0.0000000 4.7123890}%
\special{pn 8}%
\special{ar 2530 1580 50 274  4.7711448 6.2831853}%
\special{ar 2530 1580 50 274  0.0000000 4.7123890}%
%
\special{pn 8}%
\special{ar 2954 1588 38 212  4.7804654 6.2831853}%
\special{ar 2954 1588 38 212  0.0000000 4.7123890}%
\special{pn 8}%
\special{pa 3188 1598}%
\special{pa 3212 1620}%
\special{pa 3240 1636}%
\special{pa 3270 1644}%
\special{pa 3302 1650}%
\special{pa 3336 1652}%
\special{pa 3366 1646}%
\special{pa 3392 1628}%
\special{pa 3414 1604}%
\special{pa 3426 1574}%
\special{pa 3426 1572}%
\special{sp}%
\special{pn 8}%
\special{pa 3224 1622}%
\special{pa 3252 1606}%
\special{pa 3282 1596}%
\special{pa 3314 1592}%
\special{pa 3346 1592}%
\special{pa 3380 1590}%
\special{pa 3410 1596}%
\special{pa 3420 1598}%
\special{sp}%
\special{pn 8}%
\special{ar 2530 1580 50 274  4.7711448 6.2831853}%
\special{ar 2530 1580 50 274  0.0000000 4.7123890}%
\special{pn 8}%
\special{pa 2534 1306}%
\special{pa 2514 1282}%
\special{pa 2494 1256}%
\special{pa 2472 1232}%
\special{pa 2448 1210}%
\special{pa 2424 1192}%
\special{pa 2396 1174}%
\special{pa 2368 1160}%
\special{pa 2338 1146}%
\special{pa 2308 1134}%
\special{pa 2278 1130}%
\special{pa 2246 1132}%
\special{pa 2214 1134}%
\special{pa 2182 1132}%
\special{pa 2150 1132}%
\special{pa 2118 1134}%
\special{pa 2086 1136}%
\special{pa 2054 1140}%
\special{pa 2022 1146}%
\special{pa 1990 1152}%
\special{pa 1960 1160}%
\special{pa 1928 1170}%
\special{pa 1898 1180}%
\special{pa 1870 1194}%
\special{pa 1842 1208}%
\special{pa 1814 1226}%
\special{pa 1788 1244}%
\special{pa 1762 1262}%
\special{pa 1738 1284}%
\special{pa 1716 1306}%
\special{pa 1696 1334}%
\special{pa 1682 1362}%
\special{pa 1674 1392}%
\special{pa 1668 1424}%
\special{pa 1666 1456}%
\special{pa 1664 1488}%
\special{pa 1664 1520}%
\special{pa 1664 1552}%
\special{pa 1664 1584}%
\special{pa 1664 1616}%
\special{pa 1664 1648}%
\special{pa 1662 1682}%
\special{pa 1662 1714}%
\special{pa 1666 1746}%
\special{pa 1676 1774}%
\special{pa 1692 1800}%
\special{pa 1714 1826}%
\special{pa 1736 1850}%
\special{pa 1758 1878}%
\special{pa 1778 1902}%
\special{pa 1800 1926}%
\special{pa 1824 1944}%
\special{pa 1852 1958}%
\special{pa 1882 1964}%
\special{pa 1914 1968}%
\special{pa 1948 1972}%
\special{pa 1980 1976}%
\special{pa 2012 1982}%
\special{pa 2044 1990}%
\special{pa 2074 1996}%
\special{pa 2106 2002}%
\special{pa 2136 2004}%
\special{pa 2168 2004}%
\special{pa 2200 2004}%
\special{pa 2232 2002}%
\special{pa 2268 2002}%
\special{pa 2302 2002}%
\special{pa 2334 2000}%
\special{pa 2364 1994}%
\special{pa 2386 1982}%
\special{pa 2408 1964}%
\special{pa 2432 1938}%
\special{pa 2466 1902}%
\special{pa 2504 1866}%
\special{pa 2530 1842}%
\special{pa 2530 1848}%
\special{pa 2528 1852}%
\special{sp}%
\special{pn 8}%
\special{ar 2530 1580 50 274  4.7711448 6.2831853}%
\special{ar 2530 1580 50 274  0.0000000 4.7123890}%
\special{pn 8}%
\special{pa 1894 1552}%
\special{pa 1920 1572}%
\special{pa 1946 1588}%
\special{pa 1976 1602}%
\special{pa 2006 1614}%
\special{pa 2036 1622}%
\special{pa 2068 1626}%
\special{pa 2100 1626}%
\special{pa 2132 1628}%
\special{pa 2164 1628}%
\special{pa 2196 1624}%
\special{pa 2228 1616}%
\special{pa 2256 1602}%
\special{pa 2282 1584}%
\special{pa 2306 1556}%
\special{pa 2308 1526}%
\special{pa 2302 1522}%
\special{sp}%
\special{pn 8}%
\special{pa 1966 1596}%
\special{pa 1992 1578}%
\special{pa 2022 1564}%
\special{pa 2052 1556}%
\special{pa 2084 1556}%
\special{pa 2116 1554}%
\special{pa 2148 1552}%
\special{pa 2180 1554}%
\special{pa 2212 1562}%
\special{pa 2240 1564}%
\special{pa 2276 1582}%
\special{pa 2272 1582}%
\special{sp}%
%
\special{pn 8}%
\special{ar 2530 1580 50 274  4.7711448 6.2831853}%
\special{ar 2530 1580 50 274  0.0000000 4.7123890}%
%
\special{pn 8}%
\special{ar 2530 1580 50 274  4.7711448 6.2831853}%
\special{ar 2530 1580 50 274  0.0000000 4.7123890}%
%
\special{pn 8}%
\special{ar 2684 1588 44 276  1.5707963 4.7123890}%
%
\special{pn 8}%
\special{ar 2530 1580 50 274  4.7711448 6.2831853}%
\special{ar 2530 1580 50 274  0.0000000 4.7123890}%
\special{pn 8}%
\special{ar 2530 1580 50 274  4.7711448 6.2831853}%
\special{ar 2530 1580 50 274  0.0000000 4.7123890}%
\special{pn 8}%
\special{ar 3084 1584 24 202  1.3909428 4.8845798}%
\special{pn 8}%
\special{ar 2530 1580 50 274  4.7711448 6.2831853}%
\special{ar 2530 1580 50 274  0.0000000 4.7123890}%
\special{pn 8}%
\special{ar 2530 1580 50 274  4.7711448 6.2831853}%
\special{ar 2530 1580 50 274  0.0000000 4.7123890}%
\special{pn 8}%
\special{ar 2954 1588 38 212  4.7804654 6.2831853}%
\special{ar 2954 1588 38 212  0.0000000 4.7123890}%
\special{pn 8}%
\special{ar 2530 1580 50 274  4.7711448 6.2831853}%
\special{ar 2530 1580 50 274  0.0000000 4.7123890}%
\special{pn 8}%
\special{pa 2534 1306}%
\special{pa 2514 1282}%
\special{pa 2494 1256}%
\special{pa 2472 1232}%
\special{pa 2448 1210}%
\special{pa 2424 1192}%
\special{pa 2396 1174}%
\special{pa 2368 1160}%
\special{pa 2338 1146}%
\special{pa 2308 1134}%
\special{pa 2278 1130}%
\special{pa 2246 1132}%
\special{pa 2214 1134}%
\special{pa 2182 1132}%
\special{pa 2150 1132}%
\special{pa 2118 1134}%
\special{pa 2086 1136}%
\special{pa 2054 1140}%
\special{pa 2022 1146}%
\special{pa 1990 1152}%
\special{pa 1960 1160}%
\special{pa 1928 1170}%
\special{pa 1898 1180}%
\special{pa 1870 1194}%
\special{pa 1842 1208}%
\special{pa 1814 1226}%
\special{pa 1788 1244}%
\special{pa 1762 1262}%
\special{pa 1738 1284}%
\special{pa 1716 1306}%
\special{pa 1696 1334}%
\special{pa 1682 1362}%
\special{pa 1674 1392}%
\special{pa 1668 1424}%
\special{pa 1666 1456}%
\special{pa 1664 1488}%
\special{pa 1664 1520}%
\special{pa 1664 1552}%
\special{pa 1664 1584}%
\special{pa 1664 1616}%
\special{pa 1664 1648}%
\special{pa 1662 1682}%
\special{pa 1662 1714}%
\special{pa 1666 1746}%
\special{pa 1676 1774}%
\special{pa 1692 1800}%
\special{pa 1714 1826}%
\special{pa 1736 1850}%
\special{pa 1758 1878}%
\special{pa 1778 1902}%
\special{pa 1800 1926}%
\special{pa 1824 1944}%
\special{pa 1852 1958}%
\special{pa 1882 1964}%
\special{pa 1914 1968}%
\special{pa 1948 1972}%
\special{pa 1980 1976}%
\special{pa 2012 1982}%
\special{pa 2044 1990}%
\special{pa 2074 1996}%
\special{pa 2106 2002}%
\special{pa 2136 2004}%
\special{pa 2168 2004}%
\special{pa 2200 2004}%
\special{pa 2232 2002}%
\special{pa 2268 2002}%
\special{pa 2302 2002}%
\special{pa 2334 2000}%
\special{pa 2364 1994}%
\special{pa 2386 1982}%
\special{pa 2408 1964}%
\special{pa 2432 1938}%
\special{pa 2466 1902}%
\special{pa 2504 1866}%
\special{pa 2530 1842}%
\special{pa 2530 1848}%
\special{pa 2528 1852}%
\special{sp}%
\special{pn 8}%
\special{ar 2530 1580 50 274  4.7711448 6.2831853}%
\special{ar 2530 1580 50 274  0.0000000 4.7123890}%
\special{pn 8}%
\special{pa 1894 1552}%
\special{pa 1920 1572}%
\special{pa 1946 1588}%
\special{pa 1976 1602}%
\special{pa 2006 1614}%
\special{pa 2036 1622}%
\special{pa 2068 1626}%
\special{pa 2100 1626}%
\special{pa 2132 1628}%
\special{pa 2164 1628}%
\special{pa 2196 1624}%
\special{pa 2228 1616}%
\special{pa 2256 1602}%
\special{pa 2282 1584}%
\special{pa 2306 1556}%
\special{pa 2308 1526}%
\special{pa 2302 1522}%
\special{sp}%
\special{pn 8}%
\special{pa 1966 1596}%
\special{pa 1992 1578}%
\special{pa 2022 1564}%
\special{pa 2052 1556}%
\special{pa 2084 1556}%
\special{pa 2116 1554}%
\special{pa 2148 1552}%
\special{pa 2180 1554}%
\special{pa 2212 1562}%
\special{pa 2240 1564}%
\special{pa 2276 1582}%
\special{pa 2272 1582}%
\special{sp}%
\special{pn 8}%
\special{ar 2530 1580 50 274  4.7711448 6.2831853}%
\special{ar 2530 1580 50 274  0.0000000 4.7123890}%
\special{pn 8}%
\special{ar 2530 1580 50 274  4.7711448 6.2831853}%
\special{ar 2530 1580 50 274  0.0000000 4.7123890}%
\special{pn 8}%
\special{ar 2684 1588 44 276  1.5707963 4.7123890}%
\special{pn 8}%
\special{ar 2530 1580 50 274  4.7711448 6.2831853}%
\special{ar 2530 1580 50 274  0.0000000 4.7123890}%
\special{pn 8}%
\special{ar 2530 1580 50 274  4.7711448 6.2831853}%
\special{ar 2530 1580 50 274  0.0000000 4.7123890}%
\special{pn 8}%
\special{ar 3084 1584 24 202  1.3909428 4.8845798}%
\special{pn 8}%
\special{ar 2530 1580 50 274  4.7711448 6.2831853}%
\special{ar 2530 1580 50 274  0.0000000 4.7123890}%
\special{pn 8}%
\special{pa 3084 1390}%
\special{pa 3098 1418}%
\special{pa 3102 1450}%
\special{pa 3104 1482}%
\special{pa 3108 1514}%
\special{pa 3110 1546}%
\special{pa 3110 1578}%
\special{pa 3110 1608}%
\special{pa 3108 1642}%
\special{pa 3108 1674}%
\special{pa 3106 1708}%
\special{pa 3100 1738}%
\special{pa 3084 1766}%
\special{pa 3078 1774}%
\special{sp -0.045}%
\special{pn 8}%
\special{pa 2684 1326}%
\special{pa 2690 1358}%
\special{pa 2694 1390}%
\special{pa 2698 1422}%
\special{pa 2698 1454}%
\special{pa 2698 1486}%
\special{pa 2698 1518}%
\special{pa 2698 1550}%
\special{pa 2700 1582}%
\special{pa 2702 1614}%
\special{pa 2702 1646}%
\special{pa 2702 1678}%
\special{pa 2700 1710}%
\special{pa 2696 1742}%
\special{pa 2694 1774}%
\special{pa 2690 1806}%
\special{pa 2684 1836}%
\special{pa 2678 1858}%
\special{sp -0.045}%
\special{pn 8}%
\special{ar 2530 1580 50 274  4.7711448 6.2831853}%
\special{ar 2530 1580 50 274  0.0000000 4.7123890}%
\special{pn 8}%
\special{pa 2534 1306}%
\special{pa 2514 1282}%
\special{pa 2494 1256}%
\special{pa 2472 1232}%
\special{pa 2448 1210}%
\special{pa 2424 1192}%
\special{pa 2396 1174}%
\special{pa 2368 1160}%
\special{pa 2338 1146}%
\special{pa 2308 1134}%
\special{pa 2278 1130}%
\special{pa 2246 1132}%
\special{pa 2214 1134}%
\special{pa 2182 1132}%
\special{pa 2150 1132}%
\special{pa 2118 1134}%
\special{pa 2086 1136}%
\special{pa 2054 1140}%
\special{pa 2022 1146}%
\special{pa 1990 1152}%
\special{pa 1960 1160}%
\special{pa 1928 1170}%
\special{pa 1898 1180}%
\special{pa 1870 1194}%
\special{pa 1842 1208}%
\special{pa 1814 1226}%
\special{pa 1788 1244}%
\special{pa 1762 1262}%
\special{pa 1738 1284}%
\special{pa 1716 1306}%
\special{pa 1696 1334}%
\special{pa 1682 1362}%
\special{pa 1674 1392}%
\special{pa 1668 1424}%
\special{pa 1666 1456}%
\special{pa 1664 1488}%
\special{pa 1664 1520}%
\special{pa 1664 1552}%
\special{pa 1664 1584}%
\special{pa 1664 1616}%
\special{pa 1664 1648}%
\special{pa 1662 1682}%
\special{pa 1662 1714}%
\special{pa 1666 1746}%
\special{pa 1676 1774}%
\special{pa 1692 1800}%
\special{pa 1714 1826}%
\special{pa 1736 1850}%
\special{pa 1758 1878}%
\special{pa 1778 1902}%
\special{pa 1800 1926}%
\special{pa 1824 1944}%
\special{pa 1852 1958}%
\special{pa 1882 1964}%
\special{pa 1914 1968}%
\special{pa 1948 1972}%
\special{pa 1980 1976}%
\special{pa 2012 1982}%
\special{pa 2044 1990}%
\special{pa 2074 1996}%
\special{pa 2106 2002}%
\special{pa 2136 2004}%
\special{pa 2168 2004}%
\special{pa 2200 2004}%
\special{pa 2232 2002}%
\special{pa 2268 2002}%
\special{pa 2302 2002}%
\special{pa 2334 2000}%
\special{pa 2364 1994}%
\special{pa 2386 1982}%
\special{pa 2408 1964}%
\special{pa 2432 1938}%
\special{pa 2466 1902}%
\special{pa 2504 1866}%
\special{pa 2530 1842}%
\special{pa 2530 1848}%
\special{pa 2528 1852}%
\special{sp}%
\special{pn 8}%
\special{ar 2530 1580 50 274  4.7711448 6.2831853}%
\special{ar 2530 1580 50 274  0.0000000 4.7123890}%
\special{pn 8}%
\special{pa 1894 1552}%
\special{pa 1920 1572}%
\special{pa 1946 1588}%
\special{pa 1976 1602}%
\special{pa 2006 1614}%
\special{pa 2036 1622}%
\special{pa 2068 1626}%
\special{pa 2100 1626}%
\special{pa 2132 1628}%
\special{pa 2164 1628}%
\special{pa 2196 1624}%
\special{pa 2228 1616}%
\special{pa 2256 1602}%
\special{pa 2282 1584}%
\special{pa 2306 1556}%
\special{pa 2308 1526}%
\special{pa 2302 1522}%
\special{sp}%
\special{pn 8}%
\special{pa 1966 1596}%
\special{pa 1992 1578}%
\special{pa 2022 1564}%
\special{pa 2052 1556}%
\special{pa 2084 1556}%
\special{pa 2116 1554}%
\special{pa 2148 1552}%
\special{pa 2180 1554}%
\special{pa 2212 1562}%
\special{pa 2240 1564}%
\special{pa 2276 1582}%
\special{pa 2272 1582}%
\special{sp}%
\special{pn 8}%
\special{ar 2530 1580 50 274  4.7711448 6.2831853}%
\special{ar 2530 1580 50 274  0.0000000 4.7123890}%
\special{pn 8}%
\special{ar 2530 1580 50 274  4.7711448 6.2831853}%
\special{ar 2530 1580 50 274  0.0000000 4.7123890}%
\special{pn 8}%
\special{ar 2684 1588 44 276  1.5707963 4.7123890}%
\special{pn 8}%
\special{ar 2530 1580 50 274  4.7711448 6.2831853}%
\special{ar 2530 1580 50 274  0.0000000 4.7123890}%
\special{pn 8}%
\special{ar 2530 1580 50 274  4.7711448 6.2831853}%
\special{ar 2530 1580 50 274  0.0000000 4.7123890}%
\special{pn 8}%
\special{ar 3084 1584 24 202  1.3909428 4.8845798}%
\special{pn 8}%
\special{ar 2530 1580 50 274  4.7711448 6.2831853}%
\special{ar 2530 1580 50 274  0.0000000 4.7123890}%
\special{pn 8}%
\special{ar 2530 1580 50 274  4.7711448 6.2831853}%
\special{ar 2530 1580 50 274  0.0000000 4.7123890}%
\special{pn 8}%
\special{ar 2954 1588 38 212  4.7804654 6.2831853}%
\special{ar 2954 1588 38 212  0.0000000 4.7123890}%
\special{pn 8}%
\special{pa 2682 1320}%
\special{pa 2710 1336}%
\special{pa 2738 1350}%
\special{pa 2766 1364}%
\special{pa 2796 1376}%
\special{pa 2826 1384}%
\special{pa 2860 1392}%
\special{pa 2894 1398}%
\special{pa 2926 1396}%
\special{pa 2954 1386}%
\special{pa 2964 1380}%
\special{sp}%
\special{pn 8}%
\special{pa 2684 1858}%
\special{pa 2714 1846}%
\special{pa 2744 1832}%
\special{pa 2772 1818}%
\special{pa 2802 1804}%
\special{pa 2830 1794}%
\special{pa 2862 1790}%
\special{pa 2894 1788}%
\special{pa 2926 1792}%
\special{pa 2950 1794}%
\special{sp}%
\special{pn 8}%
\special{ar 2530 1580 50 274  4.7711448 6.2831853}%
\special{ar 2530 1580 50 274  0.0000000 4.7123890}%
\special{pn 8}%
\special{pa 2534 1306}%
\special{pa 2514 1282}%
\special{pa 2494 1256}%
\special{pa 2472 1232}%
\special{pa 2448 1210}%
\special{pa 2424 1192}%
\special{pa 2396 1174}%
\special{pa 2368 1160}%
\special{pa 2338 1146}%
\special{pa 2308 1134}%
\special{pa 2278 1130}%
\special{pa 2246 1132}%
\special{pa 2214 1134}%
\special{pa 2182 1132}%
\special{pa 2150 1132}%
\special{pa 2118 1134}%
\special{pa 2086 1136}%
\special{pa 2054 1140}%
\special{pa 2022 1146}%
\special{pa 1990 1152}%
\special{pa 1960 1160}%
\special{pa 1928 1170}%
\special{pa 1898 1180}%
\special{pa 1870 1194}%
\special{pa 1842 1208}%
\special{pa 1814 1226}%
\special{pa 1788 1244}%
\special{pa 1762 1262}%
\special{pa 1738 1284}%
\special{pa 1716 1306}%
\special{pa 1696 1334}%
\special{pa 1682 1362}%
\special{pa 1674 1392}%
\special{pa 1668 1424}%
\special{pa 1666 1456}%
\special{pa 1664 1488}%
\special{pa 1664 1520}%
\special{pa 1664 1552}%
\special{pa 1664 1584}%
\special{pa 1664 1616}%
\special{pa 1664 1648}%
\special{pa 1662 1682}%
\special{pa 1662 1714}%
\special{pa 1666 1746}%
\special{pa 1676 1774}%
\special{pa 1692 1800}%
\special{pa 1714 1826}%
\special{pa 1736 1850}%
\special{pa 1758 1878}%
\special{pa 1778 1902}%
\special{pa 1800 1926}%
\special{pa 1824 1944}%
\special{pa 1852 1958}%
\special{pa 1882 1964}%
\special{pa 1914 1968}%
\special{pa 1948 1972}%
\special{pa 1980 1976}%
\special{pa 2012 1982}%
\special{pa 2044 1990}%
\special{pa 2074 1996}%
\special{pa 2106 2002}%
\special{pa 2136 2004}%
\special{pa 2168 2004}%
\special{pa 2200 2004}%
\special{pa 2232 2002}%
\special{pa 2268 2002}%
\special{pa 2302 2002}%
\special{pa 2334 2000}%
\special{pa 2364 1994}%
\special{pa 2386 1982}%
\special{pa 2408 1964}%
\special{pa 2432 1938}%
\special{pa 2466 1902}%
\special{pa 2504 1866}%
\special{pa 2530 1842}%
\special{pa 2530 1848}%
\special{pa 2528 1852}%
\special{sp}%
%
\special{pn 8}%
\special{ar 2530 1580 50 274  4.7711448 6.2831853}%
\special{ar 2530 1580 50 274  0.0000000 4.7123890}%
\special{pn 8}%
\special{pa 1894 1552}%
\special{pa 1920 1572}%
\special{pa 1946 1588}%
\special{pa 1976 1602}%
\special{pa 2006 1614}%
\special{pa 2036 1622}%
\special{pa 2068 1626}%
\special{pa 2100 1626}%
\special{pa 2132 1628}%
\special{pa 2164 1628}%
\special{pa 2196 1624}%
\special{pa 2228 1616}%
\special{pa 2256 1602}%
\special{pa 2282 1584}%
\special{pa 2306 1556}%
\special{pa 2308 1526}%
\special{pa 2302 1522}%
\special{sp}%
\special{pn 8}%
\special{pa 1966 1596}%
\special{pa 1992 1578}%
\special{pa 2022 1564}%
\special{pa 2052 1556}%
\special{pa 2084 1556}%
\special{pa 2116 1554}%
\special{pa 2148 1552}%
\special{pa 2180 1554}%
\special{pa 2212 1562}%
\special{pa 2240 1564}%
\special{pa 2276 1582}%
\special{pa 2272 1582}%
\special{sp}%
%
\special{pn 8}%
\special{ar 2530 1580 50 274  4.7711448 6.2831853}%
\special{ar 2530 1580 50 274  0.0000000 4.7123890}%
%
\special{pn 8}%
\special{ar 2530 1580 50 274  4.7711448 6.2831853}%
\special{ar 2530 1580 50 274  0.0000000 4.7123890}%
\special{pn 8}%
\special{ar 2684 1588 44 276  1.5707963 4.7123890}%
\special{pn 8}%
\special{ar 2530 1580 50 274  4.7711448 6.2831853}%
\special{ar 2530 1580 50 274  0.0000000 4.7123890}%
\special{pn 8}%
\special{ar 2530 1580 50 274  4.7711448 6.2831853}%
\special{ar 2530 1580 50 274  0.0000000 4.7123890}%
\special{pn 8}%
\special{ar 3084 1584 24 202  1.3909428 4.8845798}%
\special{pn 8}%
\special{ar 2530 1580 50 274  4.7711448 6.2831853}%
\special{ar 2530 1580 50 274  0.0000000 4.7123890}%
\special{pn 8}%
\special{ar 2530 1580 50 274  4.7711448 6.2831853}%
\special{ar 2530 1580 50 274  0.0000000 4.7123890}%
\special{pn 8}%
\special{ar 2954 1588 38 212  4.7804654 6.2831853}%
\special{ar 2954 1588 38 212  0.0000000 4.7123890}%
\special{pn 8}%
\special{pa 3084 1384}%
\special{pa 3114 1390}%
\special{pa 3146 1394}%
\special{pa 3178 1392}%
\special{pa 3210 1386}%
\special{pa 3240 1376}%
\special{pa 3270 1360}%
\special{pa 3300 1354}%
\special{pa 3332 1358}%
\special{pa 3364 1364}%
\special{pa 3396 1366}%
\special{pa 3426 1378}%
\special{pa 3450 1398}%
\special{pa 3472 1422}%
\special{pa 3492 1448}%
\special{pa 3510 1474}%
\special{pa 3524 1502}%
\special{pa 3536 1532}%
\special{pa 3546 1564}%
\special{pa 3556 1594}%
\special{pa 3560 1626}%
\special{pa 3560 1658}%
\special{pa 3554 1688}%
\special{pa 3542 1718}%
\special{pa 3530 1750}%
\special{pa 3516 1780}%
\special{pa 3498 1806}%
\special{pa 3472 1822}%
\special{pa 3442 1832}%
\special{pa 3408 1836}%
\special{pa 3376 1840}%
\special{pa 3346 1844}%
\special{pa 3314 1848}%
\special{pa 3280 1852}%
\special{pa 3250 1850}%
\special{pa 3222 1836}%
\special{pa 3196 1814}%
\special{pa 3174 1788}%
\special{pa 3150 1766}%
\special{pa 3122 1758}%
\special{pa 3090 1768}%
\special{pa 3078 1774}%
\special{sp}%
\special{pn 8}%
\special{ar 2530 1580 50 274  4.7711448 6.2831853}%
\special{ar 2530 1580 50 274  0.0000000 4.7123890}%
\special{pn 8}%
\special{pa 2534 1306}%
\special{pa 2514 1282}%
\special{pa 2494 1256}%
\special{pa 2472 1232}%
\special{pa 2448 1210}%
\special{pa 2424 1192}%
\special{pa 2396 1174}%
\special{pa 2368 1160}%
\special{pa 2338 1146}%
\special{pa 2308 1134}%
\special{pa 2278 1130}%
\special{pa 2246 1132}%
\special{pa 2214 1134}%
\special{pa 2182 1132}%
\special{pa 2150 1132}%
\special{pa 2118 1134}%
\special{pa 2086 1136}%
\special{pa 2054 1140}%
\special{pa 2022 1146}%
\special{pa 1990 1152}%
\special{pa 1960 1160}%
\special{pa 1928 1170}%
\special{pa 1898 1180}%
\special{pa 1870 1194}%
\special{pa 1842 1208}%
\special{pa 1814 1226}%
\special{pa 1788 1244}%
\special{pa 1762 1262}%
\special{pa 1738 1284}%
\special{pa 1716 1306}%
\special{pa 1696 1334}%
\special{pa 1682 1362}%
\special{pa 1674 1392}%
\special{pa 1668 1424}%
\special{pa 1666 1456}%
\special{pa 1664 1488}%
\special{pa 1664 1520}%
\special{pa 1664 1552}%
\special{pa 1664 1584}%
\special{pa 1664 1616}%
\special{pa 1664 1648}%
\special{pa 1662 1682}%
\special{pa 1662 1714}%
\special{pa 1666 1746}%
\special{pa 1676 1774}%
\special{pa 1692 1800}%
\special{pa 1714 1826}%
\special{pa 1736 1850}%
\special{pa 1758 1878}%
\special{pa 1778 1902}%
\special{pa 1800 1926}%
\special{pa 1824 1944}%
\special{pa 1852 1958}%
\special{pa 1882 1964}%
\special{pa 1914 1968}%
\special{pa 1948 1972}%
\special{pa 1980 1976}%
\special{pa 2012 1982}%
\special{pa 2044 1990}%
\special{pa 2074 1996}%
\special{pa 2106 2002}%
\special{pa 2136 2004}%
\special{pa 2168 2004}%
\special{pa 2200 2004}%
\special{pa 2232 2002}%
\special{pa 2268 2002}%
\special{pa 2302 2002}%
\special{pa 2334 2000}%
\special{pa 2364 1994}%
\special{pa 2386 1982}%
\special{pa 2408 1964}%
\special{pa 2432 1938}%
\special{pa 2466 1902}%
\special{pa 2504 1866}%
\special{pa 2530 1842}%
\special{pa 2530 1848}%
\special{pa 2528 1852}%
\special{sp}%
%
\special{pn 8}%
\special{ar 2530 1580 50 274  4.7711448 6.2831853}%
\special{ar 2530 1580 50 274  0.0000000 4.7123890}%
\special{pn 8}%
\special{pa 1894 1552}%
\special{pa 1920 1572}%
\special{pa 1946 1588}%
\special{pa 1976 1602}%
\special{pa 2006 1614}%
\special{pa 2036 1622}%
\special{pa 2068 1626}%
\special{pa 2100 1626}%
\special{pa 2132 1628}%
\special{pa 2164 1628}%
\special{pa 2196 1624}%
\special{pa 2228 1616}%
\special{pa 2256 1602}%
\special{pa 2282 1584}%
\special{pa 2306 1556}%
\special{pa 2308 1526}%
\special{pa 2302 1522}%
\special{sp}%
\special{pn 8}%
\special{pa 1966 1596}%
\special{pa 1992 1578}%
\special{pa 2022 1564}%
\special{pa 2052 1556}%
\special{pa 2084 1556}%
\special{pa 2116 1554}%
\special{pa 2148 1552}%
\special{pa 2180 1554}%
\special{pa 2212 1562}%
\special{pa 2240 1564}%
\special{pa 2276 1582}%
\special{pa 2272 1582}%
\special{sp}%
\special{pn 8}%
\special{ar 2530 1580 50 274  4.7711448 6.2831853}%
\special{ar 2530 1580 50 274  0.0000000 4.7123890}%
%
\special{pn 8}%
\special{ar 2530 1580 50 274  4.7711448 6.2831853}%
\special{ar 2530 1580 50 274  0.0000000 4.7123890}%
\special{pn 8}%
\special{ar 2684 1588 44 276  1.5707963 4.7123890}%
\special{pn 8}%
\special{ar 2530 1580 50 274  4.7711448 6.2831853}%
\special{ar 2530 1580 50 274  0.0000000 4.7123890}%
\special{pn 8}%
\special{ar 2530 1580 50 274  4.7711448 6.2831853}%
\special{ar 2530 1580 50 274  0.0000000 4.7123890}%
\special{pn 8}%
\special{ar 3084 1584 24 202  1.3909428 4.8845798}%
\special{pn 8}%
\special{ar 2530 1580 50 274  4.7711448 6.2831853}%
\special{ar 2530 1580 50 274  0.0000000 4.7123890}%
\special{pn 8}%
\special{ar 2530 1580 50 274  4.7711448 6.2831853}%
\special{ar 2530 1580 50 274  0.0000000 4.7123890}%
\special{pn 8}%
\special{ar 2954 1588 38 212  4.7804654 6.2831853}%
\special{ar 2954 1588 38 212  0.0000000 4.7123890}%
\special{pn 8}%
\special{pa 3188 1598}%
\special{pa 3212 1620}%
\special{pa 3240 1636}%
\special{pa 3270 1644}%
\special{pa 3302 1650}%
\special{pa 3336 1652}%
\special{pa 3366 1646}%
\special{pa 3392 1628}%
\special{pa 3414 1604}%
\special{pa 3426 1574}%
\special{pa 3426 1572}%
\special{sp}%
\special{pn 8}%
\special{pa 3224 1622}%
\special{pa 3252 1606}%
\special{pa 3282 1596}%
\special{pa 3314 1592}%
\special{pa 3346 1592}%
\special{pa 3380 1590}%
\special{pa 3410 1596}%
\special{pa 3420 1598}%
\special{sp}%
\special{pn 8}%
\special{ar 2530 1580 50 274  4.7711448 6.2831853}%
\special{ar 2530 1580 50 274  0.0000000 4.7123890}%
\special{pn 8}%
\special{pa 2534 1306}%
\special{pa 2514 1282}%
\special{pa 2494 1256}%
\special{pa 2472 1232}%
\special{pa 2448 1210}%
\special{pa 2424 1192}%
\special{pa 2396 1174}%
\special{pa 2368 1160}%
\special{pa 2338 1146}%
\special{pa 2308 1134}%
\special{pa 2278 1130}%
\special{pa 2246 1132}%
\special{pa 2214 1134}%
\special{pa 2182 1132}%
\special{pa 2150 1132}%
\special{pa 2118 1134}%
\special{pa 2086 1136}%
\special{pa 2054 1140}%
\special{pa 2022 1146}%
\special{pa 1990 1152}%
\special{pa 1960 1160}%
\special{pa 1928 1170}%
\special{pa 1898 1180}%
\special{pa 1870 1194}%
\special{pa 1842 1208}%
\special{pa 1814 1226}%
\special{pa 1788 1244}%
\special{pa 1762 1262}%
\special{pa 1738 1284}%
\special{pa 1716 1306}%
\special{pa 1696 1334}%
\special{pa 1682 1362}%
\special{pa 1674 1392}%
\special{pa 1668 1424}%
\special{pa 1666 1456}%
\special{pa 1664 1488}%
\special{pa 1664 1520}%
\special{pa 1664 1552}%
\special{pa 1664 1584}%
\special{pa 1664 1616}%
\special{pa 1664 1648}%
\special{pa 1662 1682}%
\special{pa 1662 1714}%
\special{pa 1666 1746}%
\special{pa 1676 1774}%
\special{pa 1692 1800}%
\special{pa 1714 1826}%
\special{pa 1736 1850}%
\special{pa 1758 1878}%
\special{pa 1778 1902}%
\special{pa 1800 1926}%
\special{pa 1824 1944}%
\special{pa 1852 1958}%
\special{pa 1882 1964}%
\special{pa 1914 1968}%
\special{pa 1948 1972}%
\special{pa 1980 1976}%
\special{pa 2012 1982}%
\special{pa 2044 1990}%
\special{pa 2074 1996}%
\special{pa 2106 2002}%
\special{pa 2136 2004}%
\special{pa 2168 2004}%
\special{pa 2200 2004}%
\special{pa 2232 2002}%
\special{pa 2268 2002}%
\special{pa 2302 2002}%
\special{pa 2334 2000}%
\special{pa 2364 1994}%
\special{pa 2386 1982}%
\special{pa 2408 1964}%
\special{pa 2432 1938}%
\special{pa 2466 1902}%
\special{pa 2504 1866}%
\special{pa 2530 1842}%
\special{pa 2530 1848}%
\special{pa 2528 1852}%
\special{sp}%
\special{pn 8}%
\special{ar 2530 1580 50 274  4.7711448 6.2831853}%
\special{ar 2530 1580 50 274  0.0000000 4.7123890}%
\special{pn 8}%
\special{pa 1894 1552}%
\special{pa 1920 1572}%
\special{pa 1946 1588}%
\special{pa 1976 1602}%
\special{pa 2006 1614}%
\special{pa 2036 1622}%
\special{pa 2068 1626}%
\special{pa 2100 1626}%
\special{pa 2132 1628}%
\special{pa 2164 1628}%
\special{pa 2196 1624}%
\special{pa 2228 1616}%
\special{pa 2256 1602}%
\special{pa 2282 1584}%
\special{pa 2306 1556}%
\special{pa 2308 1526}%
\special{pa 2302 1522}%
\special{sp}%
\special{pn 8}%
\special{pa 1966 1596}%
\special{pa 1992 1578}%
\special{pa 2022 1564}%
\special{pa 2052 1556}%
\special{pa 2084 1556}%
\special{pa 2116 1554}%
\special{pa 2148 1552}%
\special{pa 2180 1554}%
\special{pa 2212 1562}%
\special{pa 2240 1564}%
\special{pa 2276 1582}%
\special{pa 2272 1582}%
\special{sp}%
\special{pn 8}%
\special{ar 2530 1580 50 274  4.7711448 6.2831853}%
\special{ar 2530 1580 50 274  0.0000000 4.7123890}%
\special{pn 8}%
\special{ar 2530 1580 50 274  4.7711448 6.2831853}%
\special{ar 2530 1580 50 274  0.0000000 4.7123890}%
\special{pn 8}%
\special{ar 2684 1588 44 276  1.5707963 4.7123890}%
\special{pn 8}%
\special{ar 2530 1580 50 274  4.7711448 6.2831853}%
\special{ar 2530 1580 50 274  0.0000000 4.7123890}%
\special{pn 8}%
\special{ar 2530 1580 50 274  4.7711448 6.2831853}%
\special{ar 2530 1580 50 274  0.0000000 4.7123890}%
\special{pn 8}%
\special{ar 3084 1584 24 202  1.3909428 4.8845798}%
%
\special{pn 8}%
\special{ar 2530 1580 50 274  4.7711448 6.2831853}%
\special{ar 2530 1580 50 274  0.0000000 4.7123890}%
\special{pn 8}%
\special{ar 2530 1580 50 274  4.7711448 6.2831853}%
\special{ar 2530 1580 50 274  0.0000000 4.7123890}%
\special{pn 8}%
\special{ar 2954 1588 38 212  4.7804654 6.2831853}%
\special{ar 2954 1588 38 212  0.0000000 4.7123890}%
\special{pn 8}%
\special{ar 2530 1580 50 274  4.7711448 6.2831853}%
\special{ar 2530 1580 50 274  0.0000000 4.7123890}%
\special{pn 8}%
\special{pa 2534 1306}%
\special{pa 2514 1282}%
\special{pa 2494 1256}%
\special{pa 2472 1232}%
\special{pa 2448 1210}%
\special{pa 2424 1192}%
\special{pa 2396 1174}%
\special{pa 2368 1160}%
\special{pa 2338 1146}%
\special{pa 2308 1134}%
\special{pa 2278 1130}%
\special{pa 2246 1132}%
\special{pa 2214 1134}%
\special{pa 2182 1132}%
\special{pa 2150 1132}%
\special{pa 2118 1134}%
\special{pa 2086 1136}%
\special{pa 2054 1140}%
\special{pa 2022 1146}%
\special{pa 1990 1152}%
\special{pa 1960 1160}%
\special{pa 1928 1170}%
\special{pa 1898 1180}%
\special{pa 1870 1194}%
\special{pa 1842 1208}%
\special{pa 1814 1226}%
\special{pa 1788 1244}%
\special{pa 1762 1262}%
\special{pa 1738 1284}%
\special{pa 1716 1306}%
\special{pa 1696 1334}%
\special{pa 1682 1362}%
\special{pa 1674 1392}%
\special{pa 1668 1424}%
\special{pa 1666 1456}%
\special{pa 1664 1488}%
\special{pa 1664 1520}%
\special{pa 1664 1552}%
\special{pa 1664 1584}%
\special{pa 1664 1616}%
\special{pa 1664 1648}%
\special{pa 1662 1682}%
\special{pa 1662 1714}%
\special{pa 1666 1746}%
\special{pa 1676 1774}%
\special{pa 1692 1800}%
\special{pa 1714 1826}%
\special{pa 1736 1850}%
\special{pa 1758 1878}%
\special{pa 1778 1902}%
\special{pa 1800 1926}%
\special{pa 1824 1944}%
\special{pa 1852 1958}%
\special{pa 1882 1964}%
\special{pa 1914 1968}%
\special{pa 1948 1972}%
\special{pa 1980 1976}%
\special{pa 2012 1982}%
\special{pa 2044 1990}%
\special{pa 2074 1996}%
\special{pa 2106 2002}%
\special{pa 2136 2004}%
\special{pa 2168 2004}%
\special{pa 2200 2004}%
\special{pa 2232 2002}%
\special{pa 2268 2002}%
\special{pa 2302 2002}%
\special{pa 2334 2000}%
\special{pa 2364 1994}%
\special{pa 2386 1982}%
\special{pa 2408 1964}%
\special{pa 2432 1938}%
\special{pa 2466 1902}%
\special{pa 2504 1866}%
\special{pa 2530 1842}%
\special{pa 2530 1848}%
\special{pa 2528 1852}%
\special{sp}%
\special{pn 8}%
\special{ar 2530 1580 50 274  4.7711448 6.2831853}%
\special{ar 2530 1580 50 274  0.0000000 4.7123890}%
\special{pn 8}%
\special{pa 1894 1552}%
\special{pa 1920 1572}%
\special{pa 1946 1588}%
\special{pa 1976 1602}%
\special{pa 2006 1614}%
\special{pa 2036 1622}%
\special{pa 2068 1626}%
\special{pa 2100 1626}%
\special{pa 2132 1628}%
\special{pa 2164 1628}%
\special{pa 2196 1624}%
\special{pa 2228 1616}%
\special{pa 2256 1602}%
\special{pa 2282 1584}%
\special{pa 2306 1556}%
\special{pa 2308 1526}%
\special{pa 2302 1522}%
\special{sp}%
\special{pn 8}%
\special{pa 1966 1596}%
\special{pa 1992 1578}%
\special{pa 2022 1564}%
\special{pa 2052 1556}%
\special{pa 2084 1556}%
\special{pa 2116 1554}%
\special{pa 2148 1552}%
\special{pa 2180 1554}%
\special{pa 2212 1562}%
\special{pa 2240 1564}%
\special{pa 2276 1582}%
\special{pa 2272 1582}%
\special{sp}%
\special{pn 8}%
\special{ar 2530 1580 50 274  4.7711448 6.2831853}%
\special{ar 2530 1580 50 274  0.0000000 4.7123890}%
\special{pn 8}%
\special{ar 2530 1580 50 274  4.7711448 6.2831853}%
\special{ar 2530 1580 50 274  0.0000000 4.7123890}%
\special{pn 8}%
\special{ar 2684 1588 44 276  1.5707963 4.7123890}%
\special{pn 8}%
\special{ar 2530 1580 50 274  4.7711448 6.2831853}%
\special{ar 2530 1580 50 274  0.0000000 4.7123890}%
\special{pn 8}%
\special{ar 2530 1580 50 274  4.7711448 6.2831853}%
\special{ar 2530 1580 50 274  0.0000000 4.7123890}%
\special{pn 8}%
\special{ar 3084 1584 24 202  1.3909428 4.8845798}%
\special{pn 8}%
\special{ar 2530 1580 50 274  4.7711448 6.2831853}%
\special{ar 2530 1580 50 274  0.0000000 4.7123890}%
\special{pn 8}%
\special{ar 2530 1580 50 274  4.7711448 6.2831853}%
\special{ar 2530 1580 50 274  0.0000000 4.7123890}%
\special{pn 8}%
\special{ar 2954 1588 38 212  4.7804654 6.2831853}%
\special{ar 2954 1588 38 212  0.0000000 4.7123890}%
\put(17.7900,-22.5700){\makebox(0,0)[lb]{$M_1(3 \eps)$}}%
\put(33.3300,-19.8400){\makebox(0,0){$\eps . M_2(1)$}}%
\end{picture}%

\begin{center}
 \caption[smoothing of $(M_{\eps}, g_{\eps})$]{smoothing of $(M_{\eps}, g_{\eps})$}
\end{center}

\end{figure}

\vspace{-0.8cm}

Let $\mcC_{a,b}$ be the cone $(a,b) \times \Sn$ endowed with the (conical) 
 metric $dr^2+r^2 h$.

\section{Small eigenvalues}
Let's show that $M(\eps)$ has no small eigenvalues. 
\begin{pro}\label{McG}If $1\leq p\leq n,$
There is a constant $\lambda_0>0$ such that, if $1\leq p\leq n,$
$$\lambda_\eps\neq 0\Rightarrow \lambda_\eps\geq\lambda_0.
$$
\end{pro}
\begin{proof}We shall use the  McGowan's lemma as enonciated in \cite{GP}.
Recall that this lemma, in the spirit of Mayer Vietoris theorem, gives
control of positive eigenvalues in terms of positive eigenvalues of certain 
covers with certain boundary conditions.
We use the cover $M_\eps=M_1(\eps)\cup \eps. (M_2(1)\cup\mcC_{1,2}).$ Let
$$U_1=M_1(\eps)\,\text{ and }\, U_2=\eps. (M_2(1)\cup\mcC_{1,2})
$$
then $U_{1,2}=U_1\cap U_2=\eps.\mcC_{1,2}$ and $H^{p-1}(U_1\cap U_2)=0$ for
$1< p\leq n.$

The lemma 1 of \cite{GP} asserts that, in this case and for these values of $p,$ 
the first positive eigenvalue of the Hodge-Laplace operator on exact $p$-forms of
$M_\eps$ is, up to a power of 2, bounded from below by
$$\lambda_0(\eps)=\Big((\frac{1}{\mu^p(U_1)}+\frac{1}{\mu^p(U_2)})
(\frac{\omega_{p,m}c_\rho}{\mu^{p-1}(U_{1,2})}+1)
\Big)^{-1}
$$ 
where $\mu^k(U)$ is the first positive eigenvalue of the Laplacian acting on
exact $k$-forms of $U$ and satisfying absolute boundary conditions, $\omega_{p,m}$
is a combinatorial constant and $c_\rho$ is the square of an upper bound of the first 
derivative of a partition of 1 subordinate to the cover.

For us $c_\rho, \mu^p(U_2)$ and $\mu^{p-1}(U_{1,2})$ are all of order $\eps^{-2},$ but
$\mu^p(U_1)$ is bounded for $p\leq n$ as was shown in \cite{AC} (remark that the
small eigenvalue exibited here in degree $m-1$ is in the coexact spectrum).
This give a uniform bound for the exact spectrum of degree $p$
with $1<p\leq n$ but the exact spectrum for 1-forms
comes from the spectrum on function which has been studied in \cite{T}, thus
the exact spectrum is controled for $1\leq p\leq n$, by Hodge duality it gives
a control for all the positive spectrum in these degrees. Finally we can assert
that there exists $\lambda_0>0$ such that 
$\forall \eps,\, \lambda_0(\eps)>\lambda_0.$
\end{proof}
The proof of the main Proposition \ref{B} needs some useful notations and estimates, it is
the goal of the following section.

\section{Estimates and tools}

As in \cite{ACP} we use the following change of variables~: with
$${\phi_\eps}_{|M_1(\eps)} =\phi_{1,\eps} \;\text{ and } \;
{\phi_\eps}_{|M_2(1)}=\eps^{p-m/2}\phi_{2,\eps}$$ 
we write on the cone
$$\phi_{1,\eps} =dr \wedge r^{-(n/2-p+1)}\beta_{1,\eps} + 
              r^{-(n/2-p)}\alpha_{1,\eps} 
$$
and define $\sigma_1=(\beta_1,\alpha_1)=U(\phi_1).$

On the other part, it is more convenient to define $r=1-s$ for $s\in[0,1/2]$ and write
$\phi_{2,\eps} =(dr\wedge r^{-(n/2-p+1)} \beta_{2,\eps}+
r^{-(n/2-p)}\alpha_{2,\eps})$ near the boundary. Then we can define, for $r\in[1/2,1]$ 
(the boundary of $M_2(1)$ corresponds to $r=1$)
$$
\sigma_2(r)=(\beta_2(r),\alpha_2(r))=U(\phi_2).
$$

The $L_2$ norm, for a form supported on $M_1$ in the cone $\mcC_{\eps,1}$, has 
the expression
$$
\|\phi\|^2=\int_{M_1}|\sigma_1|^2 dr\wedge d\vol_{\Sn}+\int_{M_2}|\phi_2|^2
d\vol_{M_2}
$$
and the quadratic form on study is
\begin{equation}\label{quadform}
q(\phi)=\int_{M(\eps)}|(d+d^\ast)\phi|^2=\int_{M_1(\eps)}|UD_1U^\ast(\sigma_1)|^2+
\frac{1}{\eps^2}\int_{M_2(1)}|D_2(\phi_{2})|^2
\end{equation}
where $D_1$, resp. $D_2$, are the Gau\ss-Bonnet operator of $M_1$, resp. $M_2$, namely
$D_j=d+d^\ast$ acting on differential forms. In terms of 
$\sigma_1,$ which, a priori, belongs to 
$C^\infty([\eps,1[,C^\infty(\Lambda^{p-1}T^\ast\Sn)\oplus C^\infty(\Lambda^{p}T^\ast\Sn))$ 
the operator has, on the cone of $M_1,$ the expression
$$
 UD_1U^\ast =\begin{pmatrix} 0&1\\
                               -1&0   
                \end{pmatrix}
       \Big( \partial_r+\frac 1 r A\Big) \;\text{ with }\; A=
                 \begin{pmatrix} \dfrac n 2 -P&-D_0\\
                                         -D_0 & P-\dfrac n 2 \end{pmatrix} 
 $$
where $P$ is the operator of degree which multiplies by $p$ a $p$-form, and $D_0$ is the 
Gau\ss-Bonnet operator of the sphere $\Sn$.

While the Hodge-deRham operator has, in these coordinates, the expression
\begin{equation}\label{laplace}
U\Delta_1U^\ast=-\partial_r^2+\frac{1}{r^2}A(A+1).
\end{equation}

The same expressions are valid for $UD_2U^\ast$ and $U\Delta_2U^\ast$ near the boundary 
of $M_2(1)$ but we shall not use them because we need global estimates on this part.

The compatibility condition is, for the quadratic form,
$\eps^{1/2}\alpha_1(\eps)=\alpha_2(1)$ and $\eps^{1/2}\beta_1=\beta_2(1)$ or
\begin{equation}\label{recol0}
\sigma_2(1)=\eps^{1/2}\sigma_1(\eps) .
\end{equation}
The compatibility condition for the Hodge-deRham operator, of first order, is obtained 
by expressing that
$D\phi\sim (UD_1U^\ast\sigma_1,\frac 1\eps UD_2U^\ast\sigma_2)$ belongs to the 
domain of $D$. In terms of $\sigma$ it gives 
\begin{equation}\label{recol1}
\sigma_2'(1)=\eps^{3/2}\sigma_1'(\eps).
\end{equation}

Let $\xi_1$ be a cut-off function on $M_1$ around $p_0$:
$$0\leq r\leq 1/2\Rightarrow\xi_1(r)=1\hbox{ and }r\geq 1\Rightarrow\xi_1(r)=0.
$$
\begin{pro}For our given family $\phi_{\eps}$ satisfying 
$\Delta (\phi_{\eps})=\lambda_{\eps}\phi_{\eps}$ with $\lambda_\eps$ bounded,
the family $(1-\xi_1).\phi_{1,\eps}$ is bounded in $H^1(M_1).$
\end{pro}  
Then it remains to study $\xi_1.\phi_{1,\eps}$ which can be expressed with the polar 
coordinates, this is the goal of the next section.
\begin{rem}The same cannot be done with the componant on $M_2$ or more precisely
this does not give what we want to prove, namely that this componant 
goes to 0 with $\eps$. To do so we have first to concider $\phi_{2,\eps}$ in the
domain of an elliptic operator, this is the main difficulty, in contrast with the 
case concerning functions. In fact we will decompose $\phi_{2,\eps}$ in a part which
clearly goes to 0 and an other part which belongs to the domain of an elliptic operator,
this operator is naturally $D_2$ but the point is to determine the boundary conditions.
\end{rem}

\subsection{Expression of the quadratic form}\tir For any $\phi$ such 
that the componant $\phi_{1}$ is supported in the cone $\mcC_{1,\eps}$, 
one has, with  $\sigma_1=U\phi_1$ and by the same calculus as in \cite{ACP}~: 
\begin{align*}
  \int_{\mcC_{\eps,1}} |D_1\phi|^2 d \vol_{g_\eps}
    &= \int_\eps^1 
     \left| \Bigl(\partial_r + \frac{1}{r} A \Bigr) 
            \sigma_1\right|^2 dr\\
    &= \int_\eps^1
     \Bigl[\, |\sigma_1'|^2+\frac{2}{r} \iprod {\sigma_1'} {A\sigma_1}
             +\frac{1}{r^2} |A\sigma_1|^2\,\Bigr] dr\\
    &= \int_\eps^1
     \Bigl[\, |\sigma_1'|^2+\partial_r 
           \Bigl( \frac{1}{r} \iprod {\sigma_1} {A\sigma_1} \Bigr)
       + \frac{1}{r^2} \bigl( \iprod {\sigma_1} {A\sigma_r} 
       + |A\sigma_1|^2 \bigr)\,\Bigr] dr\\
    &= \int_\eps^1
     \Bigl[\, |\sigma_1'|^2+\frac{1}{r^2}
              \iprod{\sigma_1} {(A+A^2)\sigma_1} \,\Bigr] dr
       - \frac{1}{\eps} \bigiprod {\sigma_1(\eps)}{A\sigma_1(\eps)}.
\end{align*}
we then have
\begin{multline}\label{fq1}
q(\phi)=\int_\eps^1
     \Bigl[\, |\sigma_1'|^2+\frac{1}{r^2}
              \iprod{\sigma_1} {(A+A^2)\sigma_1} \,\Bigr] dr
- \frac{1}{\eps} \bigiprod {\sigma_1(\eps)}{A\sigma_1(\eps)}\\
+ \frac{1}{\eps^2}\int_{M_2(1)}|D_2\phi_2|^2
\end{multline}
On the other hand we have, as well,
\begin{align*}
  \int_{\mcC_{1/2,1}} |D_2\phi|^2 d \vol_{g_\eps}
    &= \int_{1/2}^1 
     \left| \Bigl(\partial_r + \frac{1}{r} A \Bigr) 
            \sigma_2\right|^2 dr\\
      &= \int_{1/2}^1
     \Bigl[\, |\sigma_2'|^2+\frac{1}{r^2}
              \iprod{\sigma_2} {(A+A^2)\sigma_2} \,\Bigr] dr\\
      &\hspace{1cm} +\bigiprod {\sigma_2(1)}{A\sigma_2(1)}-
        \bigiprod {\sigma_2(1/2)}{A\sigma_2(1/2)}.
\end{align*}

Thus the first boundary terms annihilate, and one has also
\begin{multline}\label{fq0}
q(\phi)=\int_\eps^1
     \Bigl[\, |\sigma_1'|^2+\frac{1}{r^2}
              \iprod{\sigma_1} {(A+A^2)\sigma_1} \,\Bigr] dr+\\
 \frac{1}{\eps^2}\int_{1/2}^1
     \Bigl[\, |\sigma_2'|^2+\frac{1}{r^2}
              \iprod{\sigma_2} {(A+A^2)\sigma_2} \,\Bigr] dr
-\frac{1}{\eps^2}\bigiprod {\sigma_2(1/2)}{A\sigma_2(1/2)}.
\end{multline}
We remark that the boundary term
$-\bigiprod {\sigma_2(1/2)}{A\sigma_2(1/2)}$ is positif if 
$\sigma_2$ belongs to the eigenspace of $A$ with negative eigenvalues.
In fact we know the spectrum of $A$:
\subsection{Spectrum of $A$}\tir It has been calculated in  \cite{BS}.
By their result, we have that the spectrum of $A$ is given
by the values $\gamma=\pm\frac 12\pm\sqrt{\mu^2+(\frac{n-1}{2}-p)^2}$ for $\mu^2$ 
covering the spectrum of $\Delta_{\Sn}$ acting on the coclosed p-forms.

Now the spectrum for the standard sphere has been calculated in \cite{GM} and
as a consequence one has $\mu^2\geq (n-p)(p+1)$ on coclosed $p$-forms, unless
$p=0$ for which we have in fact $\mu^2\geq (n-p)(p+1)$ on coexact $p$-forms 
({\it ie. } non constant functions). As a consequence
 \begin{equation*}
\mu^2+(\frac{n-1}{2}-p)^2\geq(n-p)(p+1)+(\frac{n+1}{2}-(p+1))^2=
(\frac{n+1}{2})^2
\end{equation*}
and then 
\begin{equation}\label{inf}
|\gamma|\geq\frac{n}{2}.
\end{equation}
For $p=0,$ the eigenvalues of $A$ corresponding to the constant function
are in fact $\pm\frac n2$ 
as we can see with the expression of $A$, so the minoration (\ref{inf}) is allways
valid and, in particular, $0\notin\text{Spec}(A).$
\subsubsection*{consequence}\tir The elliptic operator $A(A+1)$ is non negative 
(and positive if $n\geq 3$). Indeed $A(A+1)=(A+1/2)^2-1/4$ and the values of the 
eigenvalues of $A$ give the conclusion.

\subsection{Equations satisfied}\tir  On the cones,
$\sigma=(\sigma_1,\sigma_2)$ satisfies the equations
\begin{align}\label{eq1}
\Big(-\partial_r^2+\frac{1}{r^2}A(A+1)\Big)\sigma_1&=\lambda_\eps\sigma_1\\
\label{eq2}\Delta_2 U^\ast\sigma_2&=\eps^2\lambda_\eps U^\ast\sigma_2
\end{align}

and the compatibility conditions have been given in (\ref{recol0}) and (\ref{recol1}):
\begin{equation}\sigma_2(1)=\eps^{1/2}\sigma_1(\eps),\quad
\sigma_2'(1)=\eps^{3/2}\sigma_1'(\eps).
\end{equation}

We decompose $\sigma_1$ along a base of eigenvectors of $A$~: 
$\sigma_1=\sum\sigma_1^\gamma $
and $A\sigma_1^\gamma =\gamma\sigma_1^\gamma .$

\subsection{Boundary control}\tir 
We know that $\int_\eps^1|(\partial_r+\frac {A}{r})\sigma_1|^2\leq \lambda+1$
for $\eps$ small enough.This inequality stays valid for ${\xi_1}\sigma_1$ with a 
bigger constant: there exists $\Lambda>0$ such that for any $\eps>0$

$$\sum_{\gamma\in\text{Spec}(A)}\int_\eps^1|\partial_r({\xi_1}\sigma_1^\gamma)+
\frac {\gamma}{r}({\xi_1}\sigma_1^\gamma)|^2\leq \Lambda.
$$
Then, if we remark that
$\partial_r\sigma+\frac {\gamma}{r}\sigma=r^{-\gamma}\partial_r(r^\gamma\sigma)$
we can write, for $\gamma<0\Rightarrow\gamma\leq-\frac n2 ,$ 
\begin{equation}
(\eps^\gamma\sigma_1^\gamma(\eps))^2=
\Big(\int_\eps^1\partial_r(r^\gamma{\xi_1}\sigma_1^\gamma)\Big)^2
\leq \int_\eps^1 r^{2\gamma}\int_\eps^1|\partial_r({\xi_1}\sigma_1^\gamma)+
\frac {\gamma}{r}({\xi_1}\sigma_1^\gamma)|^2
\end{equation}
So $\sigma_1^\gamma(\eps)=O(\eps^{1/2}/\sqrt{|2\gamma+1|}).$
This suggests that
the limit $\sigma$ is harmonic on $M_2(1)$ with boundary condition
$\Pi_{<0}\sigma_2=0,$ 
if $\Pi_{<0}$ denote the spectral projector of $A$ on the total eigenspace
of negative eigenvalues. The limit problem appearing here has a boundary 
condition of Atiyah-Patodi-Singer type \cite {APS}.
Indeed we have
\begin{pro}\label{bord}
There exists a constant $C$ such  that the boundary value satisfies, for all $\eps>0$
$$\|\Pi_{<0}\Big(\sigma_{1,\eps}(\eps)\Big)\|^2\leq C \eps.
$$
\end{pro}
\begin{proof} We know that $q({\xi_1}\phi_{1,\eps},\phi_{2,\eps})$
is bounded by $\Lambda,$ on the other hand the expression
of the quadratic form (\ref{fq1}) can be done with respect to the 
decomposition along $\Im\Pi_{>0}$ and $\Im\Pi_{<0}.$ Namely:
\begin{align*}q({\xi_1}\phi_{1,\eps},\phi_{2,\eps})
&=\int_\eps^1 \left| \Bigl(\partial_r + \frac{1}{r} A \Bigr)\Pi_{<0}
(\xi_1 \sigma_{1,\eps})\right|^2 dr\\
&\hspace{2cm}+\int_\eps^1 \left| \Bigl(\partial_r + \frac{1}{r} A \Bigr)\Pi_{>0}
(\xi_1 \sigma_{1,\eps})\right|^2 dr
+ \frac{1}{\eps^2}\int_{M_2(1)}|D_2\phi_2|^2\\
&\geq\int_\eps^1 \left| \Bigl(\partial_r + \frac{1}{r} A \Bigr)\Pi_{<0}
(\xi_1 \sigma_{1,\eps})\right|^2 dr\\
&\geq\int_\eps^1
     \Bigl[\, |\Pi_{<0}(\xi_1\sigma_{1,\eps})'|^2+\frac{1}{r^2}
\iprod{\Pi_{<0}(\xi_1\sigma_{1,\eps})} {(A+A^2)\Pi_{<0}(\xi_1\sigma_{1,\eps})} \,\Bigr] dr\\
&\hspace{2cm}- \frac{1}{\eps} \bigiprod {\Pi_{<0}\sigma_1(\eps)}
{A\circ\Pi_{<0}\sigma_1(\eps)}\\
&\geq\frac{n}{2\eps}\|\Pi_{<0}\sigma_{1,\eps}(\eps)\|^2
\end{align*}
because $A(A+1)$ is non negative and $-A\circ\Pi_{<0}\geq\frac{n}{2}.$
\end{proof}
\subsection{Limit problem}\label{lim} \tir We study here good candidates for the limit
Gau\ss-Bonnet operator. On $M_1$ the problem is clear, the question here 
is to identify the boundary conditions on $M_2(1).$ \\
$\bullet$ On $M_1$ the natural problem is the Friedrich extension of $D_1$ 
on the cone, it is not a real conical singularity and $\Delta_1=D_1^\ast\circ D_1$ is 
the usual Hodge-de Rham operator
(we can see with the expression of $\phi_1$ using the Bessel functions, see 
appendix, that
$\sum_\gamma |d_{\eps,\bargamma}|^2\eps^{-2\bargamma+1}/2\bargamma-1$
is bounded so $\lim_{\eps\to 0}\sum_\gamma |d_{\eps,\bargamma}|^2=0$ 
and the limit $U\phi$ has only regular components, {\it ie.} in terms of 
$f_\bargamma(r)$).
\\$\bullet$ For $n\geq 2$ the  forms on $M_2(1)$ satisfying $D_2(\phi)=0$
$\Pi_{<0}\circ U(\phi)=0$ on the boundary are precisely the $L_2$  
forms in $\Ker(D_2)$ on the large manifold $\widetilde{M_2}$  obtained 
from $M_2(1)$ by gluing a conic cylinder $[1,\infty[\times \Sn$ with metric 
$dr^2+r^2 h$, {\it ie.} the exterior of the sphere in $\R^{n+1}.$

Indeed, these $L_2$ forms must satisfy $(\partial_r+\frac 1r A)\sigma=0$ 
or, $\forall\gamma\in\text{Spec}(A),\,\exists
\sigma_0^\gamma\in \ker(A-\gamma)$ such that 
$\sigma^\gamma=r^{-\gamma} \sigma_0^\gamma\in L_2$ which is possible only
for $\gamma>1/2.$ This limit problem is of the category \emph{non parabolic
at infinity} in the terminology of Carron \cite{C}, see particularly the theorem
2.1 there, then as a consequence of theorem 0.4 of the same paper we know
that its kernel is finite dimensional, more precisely it gives:

\begin{pro}\label{prop:D_2-elliptic}
The operator $D_2$ acting on the forms of $M_2(1),$ with
the boundary condition $\Pi_{<0}\circ U=0,$ is elliptic in the sens that
the $H_1$ norm of elements of the domain is controled by the norm of the graph. 
Let's $\mcD_2$ denote this operator.
\end{pro}
\begin{cor} The kernel of $\mcD_2$ is of finite dimension and can be
identify with a subspace of the  total space
$\sum_p H^p(M_2(1))$ of absolute cohomology.
\end{cor}
We shall see in Corollary \ref{coho2} below that this kernel is in fact the  total 
space $\sum_p H^p(M_2).$
\begin{proof}
We show that there exists a constant $C>0$ such that for each
$\phi\in H^1(\Lambda T^\ast M_2(1))$ satisfying $\Pi_{<0}\circ U(\phi)=0,$
then
$$\|\phi\|_{H^1}\leq C(\|\phi\|_{L_2}+\|D_2(\phi)\|_{L_2}).
$$
Thus $\mcD_2$ is closable. 

Denote, for such a $\phi,$ by $\tilde\phi$ its
harmonic prolongation on $\widetilde M_2.$ Then $\tilde\phi$ is in the domain of
the Dirac operator on $\widetilde M_2$ which is elliptic, it means that for each 
smooth function $f$ with compact support there exists a constant $C_f>0$
such that 
$$\forall\psi \in\dom(D_2) \quad\|f.\psi \|_{H^1}\leq C_f
( \|\psi \|_{L_2}+\|D_2(\psi) \|_{L_2})
$$
(it is the fact that an operator 'non parabolic at infinity' is continue
from its domain to $H^1_{loc}$, Theorem 1.2 of Carron)

If we apply this inequality for some $f=1$ on $M_2(1)$ and $\psi=\tilde\phi$
we obtain in particular that
$$\|\phi \|_{H^1(M_2(1))}\leq C( \|\tilde\phi \|_{L_2}+\|D_2(\tilde\phi) \|_{L_2})
$$
with $C=C_f.$ We remark first that 
$$
\|D_2(\tilde\phi) \|_{L_2(\widetilde M_2)}=\|D_2(\phi) \|_{L_2(M_2(1))}.
$$ 
Now we can write, by the use of cut-off functions, $\phi=\phi_0+\bar\phi$ with 
$\phi_0$ null near the boundary and $\bar\phi$ supported in $1/2\leq r\leq 1$. Then 
$\tilde\phi_0=0$ so, for the control of $\|\tilde\phi \|_{L_2},$ we can suppose that 
$\phi=\bar\phi.$ We write $U\phi=\sigma$ and $\sigma=\sum_\gamma \sigma^\gamma$ on the 
eigenspaces of $A.$ We have
$$\|\tilde\phi\|^2_{L_2(\R^m-B(0,1))}=\sum_{\gamma>0}\frac{1}{2\gamma-1}|\sigma^\gamma(1)|^2,
$$
now $\gamma\geq 1$ and $\sigma^\gamma(1/2)=0,$ so one has 
$\sigma^\gamma(1)=\int_{1/2}^1\partial_r(r^{\gamma}\sigma^\gamma)$ and
by Cauchy-Schwarz inequality
$$|\sigma^\gamma(1)|^2\leq\int_{1/2}^1 (r^{-\gamma}\partial_r(r^{\gamma}\sigma^\gamma))^2
\int_{1/2}^1 r^{2\gamma}
$$
or
$$|\sigma^\gamma(1)|^2\leq\|(\partial_r+\frac{1}{r}A)(\sigma^\gamma)\|^2\frac{1}{2\gamma+1}
$$
as a consequence
$$\sum_{\gamma>0}\frac{1}{2\gamma-1}|\sigma^\gamma(1)|^2\leq
\sum_{\gamma>0}\|(\partial_r+\frac{1}{r}A)(\sigma^\gamma)\|^2\frac{1}{4\gamma^2-1}\leq\|D_2(\phi)\|^2
$$
then, changing the constant, we have also
$$\|\phi \|_{H^1(M_2(1))}\leq C( \|\phi \|_{L_2(M_2(1)}+\|D_2(\phi) \|_{L_2(M_2(1))}.
$$
\end{proof}
\subsubsection*{alternative proof of the proposition, in the spirit of \cite{APS}}\tir 
To study this boundary condition it is better to write again the
$p$-form near the boundary as $\phi_{2}=dr \wedge r^{-(n/2-p+1)}\beta_2 + 
              r^{-(n/2-p)}\alpha_{2}$
with, as before, $U(\phi_{2})=\sigma_2=(\beta_2,\alpha_{2})$ .
On the cone $r\in[1/2,1],$  $UD_2 U^\ast=\partial_r+\frac 1r A$
and we can construct, as in \cite{APS} a parametrix of $D_2$ by 
gluing an interior parametrix with one constructed on the 'long' cone
$r\in]0,1]$ as follows~:

Given a form $\psi$ on $M_2(1),$ if we look for a form $\phi$ such that 
$D_2\phi=\psi,$ we write $\psi$ as the sum of two terms, the first one 
with support in the neighborood of the boundary and the second one nul
near the boundary. On the second term we apply an interior parametrix $Q_0$
of the elliptic operator $D_2$. Let's now supposes that $\phi$ is 
supported in the cone $r\in [1/2,1]$. We decompose $U\psi$ along the
eigenspaces of $A:U\psi=\sum_\gamma \psi^\gamma$ and if also
$U\phi=\sum_\gamma \phi^\gamma$, then $\phi^\gamma$ must satisfy
$$\partial_r \phi^\gamma+\frac{\gamma}{r}\phi^\gamma=
r^{-\gamma}\partial_r (r^{\gamma}\phi^\gamma)=\psi^\gamma.
$$
We take the solution
\begin{align}
\phi^\gamma=r^{-\gamma}\int_1^r \rho^\gamma \psi^\gamma(\rho)d\rho 
&\text{ if }\gamma<0\\
\phi^\gamma=r^{-\gamma}\int_0^r \rho^\gamma \psi^\gamma(\rho)d\rho 
&\text{ if }\gamma>0
\end{align}
Thus $\gamma<0\Rightarrow \phi^\gamma(1)=0.$
It is now easy to verify that $\mcD_2$ satisfies the property (SE) of
\cite{L} p. 54 (with $\rho(x)=\sqrt x$). 

This fact and the vacuity of 
Spec$(A)\cap]-1,+1[$ assure the construction of the parametrix on the cone, 
see \cite{L} and also \cite{BS} who make this construction. In fact the 
parametrix on the cone gives only $H^1$ regularity with weight function, but 
we will cut the singular point for $M_2(1)$, these results are in \cite{L}
Proposition 1.3.12 and following.
\subsection{Boundedness}\tir Recall that $A(A+1)$ is non negative.

\begin{pro}\label{phi2}
Let $\chi$ be a cut-off function supported in $[3/4,1[ $ equal
to 1 on $[7/8,1[ $ and $\sigma_{2,\eps}=U(\phi_{2,\eps}).$
The family 
$\psi_{2,\eps}=\phi_{2,\eps}-U^\ast\Big(\Pi_{<0}(\chi\sigma_{2,\eps})\Big)$ belongs 
to the domain of $\mcD_2$,  
is bounded in $H^1(M_2)$ and satisfies $\lim_{\eps\to 0}\|\psi_{2,\eps}-\phi_{2,\eps}\|=0$
and
\begin{equation}\label{D2}\lim_{\eps\to 0}\|D_2(\psi_{2,\eps}-\phi_{2,\eps})\|=O(\sqrt\eps)
\end{equation}
\end{pro}

As a consequence of this result, there exists a subsequence of $\phi_{2,\eps},$ 
which converge in $L_2$ to an harmonic form satisfying the boundary conditions of 
$\mcD_2.$
\begin{proof}We write in the following $\sigma_{2,\eps}=\sigma_{2}.$ 
It is clear that $\psi_{2,\eps}$ belongs to the domain of $\mcD_2$,
and is a bounded family for the operator norm. Thus, by ellipticity it is
also a bounded family in $H^1(M_2).$ Now
$$\|\psi_{2,\eps}-\phi_{2,\eps}\|^2\leq \int_{3/4}^1|\Pi_{<0}\sigma_2(r)|^2 dr
$$
but as a consequence of \eqref{fq0}
\begin{equation}\label{L2}|\Pi_{<0}\sigma_2(r)|^2=
2\int_{1/2}^r\la\Pi_{<0}\sigma'_2(t),\Pi_{<0}\sigma_2(t)\ra dt
+|\Pi_{<0}\sigma_2(1/2)|^2\leq 2\eps\Lambda +\eps^2\frac{2}{n}\Lambda
\end{equation}
using the inequality of Cauchy-Schwarz, the fact that the $L_2$-norm of 
$\phi_\eps$ is 1 and that $(-A\circ\Pi_{<0})\geq \frac{n}{2} .$
For the second estimate:
$$D_2(\phi_{2,\eps}-\psi_{2,\eps})=D_2U^\ast\Big(\Pi_{<0}(\chi\sigma_{2,\eps})\Big)
=\chi'U^\ast\Pi_{<0}(\sigma_{2,\eps})+\chi D_2U^\ast\Pi_{<0}(\sigma_{2,\eps})
$$
and the norm of the first term is controled by $\int_{3/4}^1|\Pi_{<0}\sigma_2(r)|^2 dr$  
which is $O(\eps)$ by the estimate(\ref{L2}) and the norm of the second term by
$\|D_2(\phi_2)\|$  which is $O(\eps)$ because $q_\eps(\phi_\eps)$ is
uniformly bounded (remark that $D_2$ preserves the orthogonal decomposition following 
$\Pi_{<0}$ and $\Pi_{>0}$ on the cone).
\end{proof}

\begin{cor}The family $\Pi_{>0}\sigma_2(1)$ is bounded in $H^{1/2}(\Sn)$
as the boundary value of $\psi_{2,\eps}.$
\end{cor}

We now define a better prolongation of $\Pi_{>0}\sigma_2(1)$ on $M_1(\eps).$
More generally  let
\begin{align}
P_\eps:\Pi_{>0}\Big(H^{1/2}(\Sn)\Big)&\to H^{1}(\mcC_{\eps,1})\\
\sigma=\sum_{\gamma\in \text{Spec}(A),\gamma>0}\sigma_\gamma&\mapsto
P_\eps(\sigma)=
\sum_{\gamma\in \text{Spec}(A),\gamma>0}\eps^{\gamma-1/2}r^{-\gamma}\sigma_\gamma.
\end{align}
We remark that there exists a constant $C$ such that
\begin{equation}\label{contL2}
\|P_\eps(\sigma)\|^2_{L_2(M_1(\eps))}\leq C \sum|\sigma_\gamma|^2=C 
\|\Pi_{>0}\sigma_2(1)\|^2_{L_2(\Sn)}
\end{equation}
and also that, if $\psi_2\in\text{Dom}\mcD_2$ and with the same cut-off function 
${\xi_1},$ which has value 1 for $0\leq r\leq 1/2$ and 0 for $r\geq 1,$ then
$\Big({\xi_1} P_\eps({\psi_2}_{|\Sn}),\psi_2\Big)$  defines through the isometries $U$
an element of $H^1(M_\eps).$
Let 
$$\tilde\psi_1:={\xi_1} P_\eps({\psi_2}_{|\Sn}).
$$
We now decompose $\phi_{1,\eps}$ as follows.
Let $${\xi_1}\phi_{1,\eps}={\xi_1}(\phi^+_\eps+\phi^-_\eps)$$ according to the decomposition of 
$\sigma_1$ along the positive or negative spectrum of $A$ on the cone. Then $\tilde\psi_1$
and $\phi^+_\eps$ have the same values on the boundary so the difference
${\xi_1}\phi^+_\eps-\tilde\psi_1$ can be viewed in $H^1(M_1)$ by a prolongation by 0 on the
ball, while the boundary value of $\phi^-_\eps$ is small. We introduce for this term the
cut-off function taken in \cite{ACP}
\begin{equation*}
    {\xi}_\eps(r)=
    \begin{cases}
          1   & \text{if $r \ge 2\sqrt{\eps}$,}\\
          \dfrac{\log(2\eps)-\log r}{\log(\sqrt{\eps})} 
              & \text{if $r \in [2\eps,2\sqrt{\eps}]$,}\\
          0   & \text{if $r \le 2\eps$.}
    \end{cases}
 \end{equation*}
\begin{lem}\label{phi-}$\lim_{\eps\to 0}\|(1-{\xi}_\eps){\xi_1}\phi^-_\eps\|_{L_2}=0.$
\end{lem}
This is a consequence of the estimate of the Proposition \ref{bord}. 
\begin{pro}\label{psi}The forms $\psi_{1,\eps}=(1-\xi_1)\phi_{1,\eps}+({\xi_1}\phi^+_\eps-\tilde\psi_1)+
\xi_\eps{\xi_1}\phi^-_\eps$ belong to $H^1(M_1)$ and define a bounded family. 
\end{pro}
\begin{proof}We will show that each term is bounded. For the first one
it is already done in Proposition \ref{McG}. For the second one, we remark that 
\begin{equation}\label{feps}
(\partial_r+\frac{A}{r})(\phi^+_\eps-\tilde\psi_1)=
(\partial_r+\frac{A}{r})(\phi^+_\eps)+\partial_r({\xi_1})P_\eps({\psi_2}_{|\Sn})
:=f_\eps
\end{equation}
and $f_\eps$ is uniformly bounded in $L_2(M_1)$ because of (\ref{contL2}).
This estimate (\ref{contL2}) shows also that the $L_2$-norm of $(\phi^+_\eps-\tilde\psi_1)$
is bounded. Thus the family $({\xi_1}\phi^+_\eps-\tilde\psi_1)$ is bounded for the
$q$-norm in $H^1(M_1)$ which is equivalent to the $H^1$-norm.\\
For the third one we use the estimate due to the expression of the quadratic
form. Expriming that $\int_{\mcC_{r,1}}|D_1({\xi_1}\phi^-)|^2$ is bounded by $\Lambda$
gives that
\begin{equation}\label{val-}
\frac{1}{r} \bigiprod {\sigma^-_1(r)}{|A|\sigma^-_1(r)}\leq\Lambda
\end{equation}
by the same argument as used for the Proposition \ref{bord}. Now 
\begin{equation*}
\|D_1(\xi_\eps{\xi_1}\phi^-_\eps)\|\leq\|\xi_\eps D_1({\xi_1}\phi^-_\eps)\|+
\||d\xi_\eps| {\xi_1}\phi^-_\eps\|\leq\|D_1({\xi_1}\phi^-_\eps)\|+
\||d\xi_\eps| {\xi_1}\phi^-_\eps\|
\end{equation*}
the first term is bounded and, with $|A|\geq\frac{n}{2}$ and the estimate (\ref{val-}), we have
\begin{align*}
\||d\xi_\eps| {\xi_1}\phi^-_\eps\|^2 
&\leq\frac{8\Lambda}{n\log^2\eps}\int_\eps^{\sqrt\eps}\frac{dr}{r}\\ 
&\leq \frac{4\Lambda}{n|\log\eps|}.
\end{align*}
This complete the proof.\end{proof}
In fact the decomposition used here is almost orthogonal:
\begin{lem}\label{lem1}
$$<(\phi^+_\eps-\tilde\psi_1),\tilde\psi_1>==O(\sqrt\eps).
$$
\end{lem}
\subsubsection*{proof of lemma \ref{lem1}}\tir 
If we decompose the terms under the eigenspaces of $A$ we see that only
the positive eigenvalues are involved and, with $f_\eps=\sum_{\gamma>0}f^\gamma$
and $(\phi^+_\eps-\tilde\psi_1)=\sum_{\gamma>0}\phi_0^\gamma,$ the equation 
(\ref{feps}) and the fact that $(\phi^+_\eps-\tilde\psi_1)(\eps)=0$ give
$$\phi_0^\gamma(r)=r^{-\gamma}\int_\eps^r \rho^\gamma f^\gamma(\rho) d\rho.
$$
Then for each positive eigenvalue $\gamma$ of $A$
\begin{align*}
<(\phi_0^\gamma,\tilde\psi_1^\gamma> &=
\eps^{\gamma-1/2}\int_\eps^1 r^{-2\gamma}\int_
\eps^r \rho^\gamma <\sigma_\gamma,f^\gamma(\rho)>_{L_2(\Sn)} d\rho\\
&=\eps^{\gamma-1/2}\int_\eps^1\frac{r^{-2\gamma+1}}{2\gamma-1}r^\gamma
<\sigma_\gamma,f^\gamma(r)>_{L_2(\Sn)}dr+\\
&\hspace{5cm}\frac{\eps^{\gamma-1/2}}{2\gamma-1}\int_
\eps^1 \rho^\gamma <\sigma_\gamma,f^\gamma(\rho)>_{L_2(\Sn)} d\rho\\
&\leq\eps^{\gamma-1/2}\int_\eps^1\frac{r^{-\gamma+1}}{2\gamma-1}
<\sigma_\gamma,f^\gamma(r)>_{L_2(\Sn)}dr+\\
&\hspace{5cm} \frac{\eps^{\gamma-1/2}}{(2\gamma-1)\sqrt{2\gamma+1}}
\|\sigma_\gamma\|\|f^\gamma\|_{L_2(\mcC_{\eps,1})}\\
&\leq C\eps^{\gamma-1/2}\|\sigma_\gamma\|\frac{\eps^{(-2\gamma+3)/2}}{(2\gamma-1)(2\gamma-1)
\sqrt{2\gamma-3}}\|f^\gamma\|_{L_2(\mcC_{\eps,1})}+\\
&\hspace{5cm}\frac{\eps^{\gamma-1/2}}{(2\gamma-1)\sqrt{2\gamma+1}}
\|\sigma_\gamma\|\|f^\gamma\|_{L_2(\mcC_{\eps,1})}\\
&\leq C\sqrt{\eps}\|\sigma_\gamma\|\|f^\gamma\|_{L_2(\mcC_{\eps,1})}.
\end{align*}

This estimate gives the lemma.\\
Remark: For $\gamma>1$, and so for $n>2$, this estimate is better.
\section{proof of theorem \ref{B}}
\begin{lem}If $\lambda\neq 0,$ then $\lambda_\eps\neq 0$ for all $\eps$ and
$$\lim_{\eps\to 0}(L_2)\tilde\psi_{1,\eps}=0$$
and also
$$\lim_{\eps\to 0}(L_2)\psi_{2,\eps}=0
$$
as well as in $q$-norm. 
\end{lem}
\begin{proof}We know, by the Proposition \ref{McG}, that there is a universal lower
bound for positive eigenvalues on $M(\eps)$, so if $\lambda=\lim\lambda_\eps$ is 
positive, it means that all the $\lambda_\eps$ are also positive! 
We know that $\psi_{2,\eps}$ is in the domain of $\mcD_2$, we decompose
$$\psi_{2,\eps}=\psi^0_{2,\eps}+\bar\psi_{2,\eps}
$$
along $\Ker\mcD_2$ and its orthogonal. Each part is bounded in $H^1(M_2(1))$ and can be
prolongated on the cone using $P_\eps$. 

Let $\tilde\psi^0_{1,\eps}={\xi_1} P_\eps({\psi^0_2}_{|\Sn}),$ 
$\bar\psi_{1,\eps}={\xi_1} P_\eps({\bar\psi}_{2|\Sn})$ and 
$$\psi_\eps=(\tilde\psi^0_{1,\eps}+\bar\psi_{1,\eps},\psi_{2,\eps}).
$$
Now
$$\psi^0_\eps=(\tilde\psi^0_{1,\eps},\psi^0_{2,\eps})\in \text{dom}(q).
$$
The $L_2$-norm of $\psi^0_\eps$ is bounded and 
\begin{align*}
q(\psi^0_\eps)&=\int_\eps^1|\xi_1'(r)P_\eps(\sigma^0_2)|^2\\
&\leq C\int_{1/2}^1|P_\eps({\psi^0_2}_{|\Sn})|^2\\
&\leq O(\eps^{n-1})
\end{align*}
due to the expression of $P_\eps$ the fact that spec$(|A|)\geq\frac n2$ and
the uniform boundedness of $P_\eps$.
Because $n\geq 2$ and  Proposition \ref{McG} is true, we conclude that the distance of 
$\psi^0_\eps$ to $\Ker\Delta_\eps$ is $O(\eps).$
But we know that $\lambda_\eps\neq 0$, so $\phi_\eps$ is orthogonal to
$\Ker\Delta_\eps$ and, with the previous result
$$<\phi_\eps,\psi^0_\eps>=O(\sqrt\eps).
$$

On the other hand we have that 
$$\int|\mcD_2\bar\psi_{2,\eps}|^2=O(\eps)\Rightarrow \|\bar\psi_{2,\eps}\|_{L_2(M_2(1))}=
O(\sqrt\eps)$$
and finally $\|\bar\psi_{2,\eps}\|_{H^1(M_2(1))}=O(\sqrt\eps)$ by ellipticity
so $\|\bar\psi_{1,\eps}\|_{L_2(M_1(\eps))}=O(\sqrt\eps)$ by uniform continuity of $P_\eps$.
and we have also
$$<\phi_\eps,\psi_\eps>=O(\sqrt\eps).
$$

Now we use Proposition \ref{phi2} and Lemma \ref{lem1}, the conclusion is
$$\lim_{\eps\to 0}\|\tilde\psi_{1,\eps}\|^2+\|\psi_{2,\eps}\|^2=0.
$$

\end{proof}

As a consequence of this result and \pref{phi2}, we obtain

\begin{cor}\label{0phi2}$\lim_{\eps\to 0}(L_2)\phi_{2,\eps}=0.$
\end{cor}

Recall now that $\psi_{1,\eps}=\phi_{1,\eps}-\tilde\psi_{1,\eps}-(1-{\xi}_\eps)\xi_1\phi^-_\eps$
and that we know, by the last Lemma and Lemma \ref{phi-}, that the two last terms 
converge to 0.
\begin{cor}
We can extract from $\psi_{1,\eps}$ a subsequence which converge in $L_2$
and weakly in $H^1,$ and any such subsequence defines at the limit a form
$\phi\in H^1(M_1)$ such that
$$\|\phi\|_{L_2}=1 \text{ and } \Delta\phi=\lambda\phi \text{ weakly.}
$$
\end{cor}
\section{proof of theorem \ref{C}}
\subsection{multiplicity of 0}
The dimension of the kernel of $\Delta_\eps$ is given by the cohomology
of $M$ which can be calculated with the Mayer-Vietoris sequence associated
to the covering $U_1,U_2$ introduced at the beginning, see \pref{McG}. 
If we remember also that 
$H^p(M_j-B,\R)\sim H^p(M_j,\R)$ for $p<m,$ we obtain that
$H^p(M,\R)\sim H^p(M_1,\R)\oplus H^p(M_2,\R)$ for $1\leq p\leq (m-1)$
while $H^{0,m}(M,\R)\sim H^{0,m}(M_1,\R)\sim H^{0,m}(M_2,\R).$

The transplantation of the harmonic forms of $M_1$ in $M$ has been described
in \cite{AC}. With the previous calculation, we have good candidates for
transplantation of the cohomology of $M_2$: for each $\sigma_2\in\Ker\mcD_2$
with $L_2$-norm equal to 1, let
$$\tilde\psi_\eps=(\tilde\psi_1,\psi_2)=U^\ast
\Big(\xi_1 P_\eps({\sigma_2}_{|\Sn}),\sigma_2\Big).
$$
Now let $\phi_\eps\in\Ker\Delta_\eps.$ We apply to $\phi_\eps$ the preceding 
estimates: there exists a subsequence which gives at the limit 
$\psi_1\in\Ker\Delta_1$ and $\psi_2\in\Ker\mcD_2;$ and only one of these
two terms can be zero. The conclusion is that all the harmonic forms of
$M_\eps$ can be approched by forms like $\tilde\psi_\eps$ or 
$\chi_\eps \phi_1,$ with $\phi_1\in \Ker\Delta_1.$
As a consequence on has
\begin{cor}\label{coho2}For $1\leq p\leq (m-1)$ the two spaces $H^p(M_2,\R)$ and
$\Ker\mcD_2$ are isomorphic.
\end{cor}
\subsection{convergence of the positive spectrum} The proof is made by induction.
We show first that $\lim\lambda_1(\eps)=\lambda_1$:
\begin{proof}We know by the Proposition \ref{Taka} that $\limsup\lambda_1(\eps)\leq\lambda_1$
and by Proposition \ref{B} that $\liminf\lambda_1(\eps)$ is in the positive spectrum of
$\Delta_1,$ and as a consequence $\liminf\lambda_1(\eps)\geq\lambda_1.$
\end{proof}
\noindent Now suppose that for all $j,\,1\leq j\leq k$ one has $\lim\lambda_j(\eps)=\lambda_j,$
we have to show that $\lim\lambda_{k+1}(\eps)=\lambda_{k+1}.$
\begin{proof}We know by Proposition \ref{Taka} that $\limsup\lambda_{k+1}(\eps)\leq\lambda_{k+1};$
let $\phi^{(1)}_\eps,\dots,\phi^{(k+1)}_\eps$ be an orthonormal family of eigenforms on $M(\eps)$:
$$\Delta_\eps\phi^{(j)}_{\eps}=\lambda_j(\eps)\phi^{(j)}_{\eps}
$$
and choose a sequence $\eps_l\to 0$ such that 
$$\lim_{l\to\infty}\lambda_{k+1}(\eps_l)=\liminf\lambda_{k+1}(\eps).$$
We apply to each $\phi^{(j)}_{\eps}$ the same decomposition as in Proposition \ref{psi},
this gives a family $\psi^{(1)}_{\eps},\dots,\psi^{(k+1)}_{\eps}$ bounded in $H^1(M_1)$ and
such that for each indice $j$ 
$$\lim_{\eps\to 0}\|\phi^{(j)}_{1,\eps}-\psi^{(j)}_{\eps}\|=0
$$
while, as in Corollary \ref{0phi2}
$$\lim_{\eps\to 0}(L_2)\phi^{(j)}_{2,\eps}=0.
$$
So, by extraction of a subsequence,we can suppose that 
$\psi^{(1)}_{\eps_l},\dots,\psi^{(k+1)}_{\eps_l}$  converge in $L_2(M_1)$ and weakly
in $H^1(M_1),$ the limit $\phi^{(1)},\dots,\phi^{(k+1)}$ is orthonormal and satisfies
$$ \forall j, \, 1\leq j\leq k \Delta_1\phi^{(j)}=\lambda_j\phi^{(j)} \hbox{ and }
\Delta_1\phi^{(k+1)}=\liminf\lambda_{k+1}(\eps)\phi^{(k+1)}.
$$
This shows that $\liminf\lambda_{k+1}(\eps)\geq\lambda_{k+1}$ and finishes the proof.
\end{proof}


\end{document}